\documentclass[12pt]{amsart} 
\usepackage{amsmath,amssymb}
\usepackage{latexsym}
\usepackage{amssymb} 
\usepackage{latexsym}
\usepackage{delarray} 
\usepackage{mathrsfs}
\usepackage{graphicx,color}

\usepackage[dvipdfmx]{hyperref}

\numberwithin{equation}{section}
\theoremstyle{plain}
\newtheorem{thm}{Theorem}[section]
\newtheorem{lem}[thm]{Lemma}
\newtheorem{prop}[thm]{Proposition}

\theoremstyle{definition}
\newtheorem{defn}[thm]{Definition}

\newcommand\R{{\mathbb R}}
\newcommand\C{{\mathbb C}}
\newcommand\N{{\mathbb N}}
\newcommand\Z{{\mathbb Z}}

\newcommand\de{\delta}
\newcommand\De{\Delta}
\newcommand\ep{\varepsilon}

\newcommand\la{\lambda}
\newcommand\va{\varphi}

\newcommand\om{\omega}
\newcommand\Om{\Omega}
\newcommand\pa{\partial}

\def\Dcal{\mathcal{D}}

\def\Qcal{\mathcal{Q}}
\def\Bscr{\mathscr{B}}
\def\Dscr{\mathscr{D}}
\def\Fscr{\mathscr{F}}
\def\Hscr{\mathscr{H}}
\def\Sscr{\mathscr{S}}

\makeatletter
\newcommand{\vertiii}[1]{{\left\vert\kern-0.25ex\left\vert\kern-0.25ex\left\vert #1 
    \right\vert\kern-0.25ex\right\vert\kern-0.25ex\right\vert}}

\title{$L^p$-mapping properties for 
Schr\"odinger operators in open sets of $\R^d$}
\author[T. Iwabuchi, T. Matsuyama, 
K. Taniguchi]{Tsukasa Iwabuchi, Tokio Matsuyama 
and Koichi Taniguchi }
\address{ 
Tsukasa Iwabuchi \endgraf 
Department of Mathematics \endgraf 
Osaka City University \endgraf
3-3-138 Sugimoto, Sumiyoshi-ku \endgraf
Osaka 558-8585 \endgraf
Japan}
\email{iwabuchi@sci.osaka-cu.ac.jp}
\address{ 
Tokio Matsuyama \endgraf 
Department of Mathematics \endgraf 
Chuo University \endgraf 
1-13-27, Kasuga, Bunkyo-ku \endgraf 
Tokyo 112-8551 \endgraf 
Japan}
\email{tokio@math.chuo-u.ac.jp} 
\address{ 
Koichi Taniguchi \endgraf 
Department of Mathematics \endgraf 
Chuo University \endgraf 
1-13-27, Kasuga, Bunkyo-ku \endgraf 
Tokyo 112-8551 \endgraf 
Japan} 
\email{koichi-t@gug.math.chuo-u.ac.jp} 
\thanks{
The first author was supported by JSPS Grant-in-Aid for Young 
Scientists (B) (No. 25800069), 
Japan Society for the Promotion of Science. 
The second author was supported by 
Grant-in-Aid for Scientific 
Research (C) (No. 15K04967), 
Japan Society for the Promotion of Science. }


\begin{document}


\begin{abstract}
Let $H_V=-\De +V$ be a Schr\"odinger operator on an arbitrary open set $\Om$ of $\R^d$,
where $d \ge 3$, and $\De$ is the Dirichlet Laplacian and the potential $V$ belongs to the Kato class on $\Om$. 
The purpose of this paper is 
to show $L^p$-boundedness of an operator $\va(H_V)$ for any rapidly decreasing function $\va$ on $\mathbb{R}$. 
$\va(H_V)$ is defined by the spectral theorem. 
As a by-product, $L^p$-$L^q$-estimates for $\va(H_V)$ are also obtained.
\end{abstract}
\maketitle


\section{Introduction and main result}
Let 
$\Omega$ be an open set of $\R^d$, where $d\ge3$. 
We consider the Schr\"odinger operator 
$$
H_V = H + V(x),
$$ 
where 
\[
H := -\De = -\sum^{d}_{j=1} \frac{\partial^2}{\partial x^2_j}
\]
is the Dirichlet Laplacian with domain 
\[
	\Dcal(H) = \left\{u \in H^1_0(\Om)\, \big| \, \De u \in L^2(\Om)\right\}
\]
and 
$V(x)$ is a real-valued measurable function on $\Om$.  
If we impose an appropriate assumption on $V(x)$, $H_V$ will be a self-adjoint operator on $L^2(\Om)$. 
Let $\{ E_{H_V}(\la) \}_{\la \in \mathbb{R}}$ be the spectral resolution of the identity for $H_V$. Then $H_V$ is written as
\[
	H_V = \int^{\infty}_{-\infty} \la \, d E_{H_V}(\la). 
\]
Hence, for any Borel measurable function $\va$ on $\R$, 
we can define an operator $\va(H_V)$ by letting
\[
	\va(H_V) = \int^{\infty}_{-\infty} \va(\la) \, d E_{H_V}(\la).
\]
This operator is initially defined on $L^2(\Om)$. 
The present paper is devoted to investigation of functional calculus for Schr\"odinger operators on $\Om$. 
More precisely, 
our purpose is to prove that 
$\va(H_V)$ is extended uniquely to a bounded linear operator on $L^p(\Om)$ for $1\le p\le \infty$ and 
that $L^p$-boundedness of $\va(\theta H_V)$ is uniform with respect to a parameter $\theta > 0$. \\

When $\Om = \R^d$, 
Simon considered the Kato class $K_{d}$ of potentials to reveal 
$L^p$-mapping properties of the Schr\"odinger operators $H_V$ and $e^{-tH_V}$ for $t>0$ 
(see \cite[Section A.2]{Si}). 
We now define a Kato class $K_d(\Om)$ on an 
open set $\Om$ as follows: 
We say that 
a real-valued measurable function $V$ on $\Om$ belongs to the class $K_{d}( \Om)$ if 
\[
	\lim_{r \rightarrow 0} \sup_{x \in \Om} \int_{\Om \cap \{|x-y|<r\}} \frac{|V(y)|}{|x-y|^{d-2}} \,dy = 0. 
\]
Throughout this paper, defining the ``Kato norm": 
\[
	\|V\|_{K_{d}(\Om)} := \sup_{x \in \Om} \int_{\Om} \frac{|V(y)|}{|x-y|^{d-2}} \,dy, 
\]
we impose an assumption on $V$ as follows: \\

\noindent{\bf Assumption A.} {\it 
Let $d \ge 3$. 
A real-valued measurable function $V(x)$ on 
$\Om$ is decomposed into $V = V_+ - V_-$, 
$V_{\pm} \ge 0$, belongs to $K_d(\Omega)$ 
and satisfies 
\begin{equation}\label{EQ:A-2}
\|V_-\|_{K_{d}(\Om)} < \frac{\pi^{d/2}}{\Gamma(d/2 - 1)},
\end{equation}
where $\Gamma(s)$ is the Gamma function for 
$s>0$.}\\

If the potential $V$ satisfies assumption A, it is proved in Proposition \ref{Prop:Self-adjoint} that  
$H_V$ is non-negative and self-adjoint on $L^2(\Om)$ (see \S 2). 
For a Borel measurable function $\va$ on $\mathbb{R}$, we define the operator $\va(H_V)$ on $L^2(\Om)$ by letting  
\begin{align*}
	\mathcal{D}(\va(H_V)) &= \Big\{ f \in L^2(\Om) \ \Big| \  \int^{\infty}_{0} |\va(\la)|^2 \,d {\langle E_{H_V}(\la)f, f \rangle}_{L^2(\Om)} < \infty \Big\}, \\
	\big{\langle\va(H_V) f , g \big\rangle}_{L^2(\Om)} &= \int^{\infty}_{0} \va(\la) \,d {\langle E_{H_V}(\la)f, g \rangle}_{L^2(\Om)},\quad \forall f \in \mathcal{D}(\va(H_V)),\ \forall g \in L^2(\Om), 
\end{align*}
where ${\langle \cdot, \cdot \rangle}_{L^2(\Om)}$ stands for the inner product in $L^2(\Om)$. 
Formally we write 
\begin{equation}\label{SD}
	\va(H_V) = \int^{\infty}_{0} \va(\la) d E_{H_V}(\la).
\end{equation}

\vspace{0.5cm}
In this paper we denote by $\mathscr{B}(X,Y)$ the space of all bounded
linear operators from a Banach space $X$ to another one $Y$.
When $X=Y$, we denote by 
$\mathscr{B}(X)=\mathscr{B}(X,X)$.
We use the notation $\mathcal{D}(T)$ for the domain of an operator $T$.\\

Denoting by $\mathscr{S}(\mathbb{R})$ 
the space of rapidly decreasing functions on $\mathbb{R}$, 
we shall prove here the following: 

\begin{thm}\label{Thm1.1}
Let $d\ge3$, $\va \in \mathscr{S}(\mathbb{R})$ and $1 \le p \le \infty$. 
Assume that the measurable potential $V$ satisfies assumption A. 
Then there exists a constant $C = C(d, \va) > 0$ such that   
\begin{equation}\label{Th1.1}
	\|\va(\theta H_V)\|_{\mathscr{B}(L^p(\Om))} \le C 
\end{equation}
for any $\theta > 0$.
\end{thm}

Let us give a few remarks on Theorem \ref{Thm1.1}. 
We have restricted the result in this theorem 
to high space dimensions.
So, one would expect the result to hold also 
for low space dimensions, 
i.e., $d=1,2$. However, in the present paper 
we use 
the pointwise estimates for kernel of $e^{-tH_V}$  
on $\R^d$ that D'Ancona and Pierfelice proved for 
$d\ge 3$ (see \cite{DP}).
Hence low dimensional cases will be a 
future problem. 
When $V=0$, 
Theorem \ref{Thm1.1} also holds in the cases 
$d = 1,2$ 
by using the pointwise estimates for 
classical heat kernel of $e^{t\De}$. 

One can easily see that $\va(H_V)$ is bounded 
on $L^2(\Om)$ 
via direct application of the spectral 
resolution \eqref{SD}. 
From the point of view of harmonic analysis, 
it would be important to obtain 
$L^p$-boundedness ($p \neq 2$). 
For instance,  
Theorem \ref{Thm1.1} 
provides a generalization of 
$L^p$-boundedness for the Fourier multiplier 
in $\R^d$: 
\[
\big\| \Fscr^{-1}\big[\hat{\va}
(\theta\, |\cdot|^2) \hat{f}\,\big] 
\big\|_{L^p(\R^d)} \le C_p 
\|f\|_{L^p(\R^d)}, \quad \forall \theta>0, 
\]
where $\va \in \Sscr({\R^d})$,  
$\hat{}$ denotes the Fourier transform, and 
$\Fscr^{-1}$ is the Fourier inverse transform. 
$L^p$-boundedness of $\va(\theta H_V)$ 
also plays a fundamental role in defining 
the Besov spaces generated by $H_V$
(see, e.g., \cite{DP,GV,JN1}). 
Thus Theorem \ref{Thm1.1} would be a starting point of the study of spectral multiplier and Besov spaces on open sets. \\

When $\Om = \R^d$, there are some known results on uniform $L^p$-estimates for $\va(\theta H_V)$ with respect to $\theta$. 
For $0 <\theta\le1$, Jensen and Nakamura proved the uniform estimates for $d \ge 1$, 
under the assumption
that the potential $V = V_+ - V_-$, $V_{\pm} \ge 0$, satisfies $V_+ \in K^{\mathrm{loc}}_d(\R^d)$ and $V_- \in K_{d}(\R^d)$ (see \cite{JN1,JN2}). 
Here $K^{\mathrm{loc}}_d(\R^d)$ is the local Kato class, which is the space of all $f \in L^1_\mathrm{loc} (\R^d)$ such that $f$ belongs to the Kato class on any compact set in $\R^d$. 
For $\theta >0$, Georgiev and Visciglia proved the uniform estimates 
under the assumption 
that the potential $V$ satisfies 
\[
0\le V(x)\le \frac{C}{|x|^2(|x|^\ep + |x|^{-\ep})}\quad (C>0,\ \ep>0)
\]
when $d=3$ (see \cite{GV}).  
D'Ancona and Pierfelice proved the uniform estimates for $d\ge3$, 
under the assumption 
that the potential $V=V_+ -V_-$, $V_{\pm} \ge 0$, satisfies $V_\pm \in K_d(\R^d)$ and 
\[
\|V_-\|_{K_d(\R^d)} < \frac{\pi^{d/2}}{\Gamma(d/2-1)}
\]
(see \cite{DP}). 
As far as we know, 
Theorem \ref{Thm1.1} is new in the sense that 
there would not be no results on $L^p$-estimates for $\va(H_V)$ in open sets. \\

Let us overview the strategy of proof of 
Theorem \ref{Thm1.1}. For the sake of 
simplicity, we consider the case $V \equiv 0$, 
since the case $V\not\equiv 0$ is similar. 
The original idea of proof of $L^p$-boundedness 
goes back to Jensen and Nakamura \cite{JN2}. 
The method for the~boundedness of 
$\varphi (-\Delta)$ 
is to use the amalgam spaces $\ell^p(L^q)$, 
pointwise estimates for the kernel of 
$e^{-t\Delta}$ 
and the commutator estimates for 
$-\Delta$ and polynomials. 
As to the uniform boundedness of 
$\varphi (-\theta \Delta)$ 
with respect to $\theta$, 
the estimates for operator 
$\varphi (-\theta \Delta)$ are reduced to 
those on $\varphi (-\Delta)$ via the 
following equality 
\begin{equation}\label{EQ:scale}
\left(\va(-\theta \De)f\right)(x)=\Big(\va(-\De)\left(f(\theta^{1/2}\cdot)\right)\Big)(\theta^{-1/2}x),\quad x\in\R^d,\quad \theta>0
\end{equation}
(see \cite{DP,GV,JN1}).
There, scaling invariance of $\R^d$, i.e., $\R^d=\theta^{1/2}\R^d$, plays an 
essential role in the argument. 
On the other hand, when one tries to get \eqref{EQ:scale} on open sets $\Om \subsetneqq \R^d$, 
the scaling invariance breaks down, i.e., $\Om\ne\theta^{1/2}\Om$. 
To avoid this problem, 
we shall introduce the scaled amalgam spaces $\ell^p(L^q)_{\theta}(\Om)$ 
to estimate the operator norm of $\varphi (-\theta \Delta)$ directly. 
A scale exponent $1/2$ in $\theta^{1/2}$ of the spaces $\ell^p(L^q)_{\theta}(\Om)$ is chosen 
to fit the scale exponent of the operator $\varphi (-\theta \Delta)$; thus 
we define the scaled amalgam spaces as follows:

\begin{defn}[Scaled amalgam spaces $\ell^p(L^q)_{\theta}$]
\label{Def1.2}
Let $1 \le p,q\le \infty$ and $\theta > 0$. 
The space $\ell^p(L^q)_{\theta}$ is defined by letting
\[
\ell^p(L^q)_{\theta} = \ell^p(L^q)_{\theta}(\Om) := \Big\{ f \in L^q_{\mathrm{loc}}(\overline{\Om}) \Big| 
\sum_{n \in \mathbb{Z}^d} \|f\|^p_{L^q(C_{\theta}(n))} < \infty \Big\}
\] 
with norm 
\[
\|f\|_{\ell^p(L^q)_{\theta}}
=\left\{
\begin{aligned}
& \Big( \sum_{n \in \mathbb{Z}^d} 
\|f\|^p_{L^q(C_{\theta}(n))} \Big)^{1/p}
& \quad \text{for $1\le p<\infty$,}\\
& \sup_{n \in \mathbb{Z}^d} 
\|f\|_{L^q(C_{\theta}(n))}
& \quad \text{for $p=\infty$,}
\end{aligned}\right.
\]
where $C_{\theta}(n)$ is the cube centered at 
$\theta^{1/2} n \in \theta^{1/2} 
\mathbb{Z}^d$ with side length $\theta^{1/2}$;
\[
C_{\theta}(n)=\Big\{ x=(x_1,x_2,\cdots,x_d) \in\Om \ \Big|
\max_{j = 1, \cdots, d} 
|x_j - \theta^{1/2} n_j| 
\le \frac{\theta^{1/2}}{2} \Big\}.
\]
Here we adopt the Euclidean norm 
for $n=(n_1,n_2,\ldots, n_d)\in\Z^d$; 
\[
|n|=\sqrt{n^2_1+n^2_2+\cdots+n^2_d}.
\]
\end{defn}
It can be checked that $\ell^p(L^q)_{\theta}$ 
is a Banach space with norm 
$\|\cdot\|_{\ell^p(L^q)_{\theta}}$ having 
the property that
\[
\ell^p(L^q)_{\theta} \hookrightarrow 
L^p(\Om) \cap L^q(\Om)
\]
for $1 \le p \le q \le \infty$. \\

This paper is organized as follows. In \S 2 
the self-adjointness of 
Schr\"odinger operator $H_V$ is shown. 
In \S 3 we prepare the pointwise estimate for 
the kernel of $e^{-t H_V}$. 
\S4 is devoted to the uniform estimates 
in $\ell^p(L^q)_\theta$ for the resolvent of 
$H_V$. In \S5 we derive 
the commutator estimates for our problem 
in the open set $\Om$. 
In \S 6 the proof of Theorem \ref{Thm1.1} 
is given. 
As a by-product of Theorem \ref{Thm1.1}, 
$L^p$-$L^q$-boundedness for $\va(H_V)$ is 
proved in \S 7.


\section{Self-adjointness of Schr\"odinger operators}
In this section we show that operator $H_V$ is self-adjoint and non-negative under 
assumption A. \\
 
Our purpose is to prove the following.

\begin{prop}\label{Prop:Self-adjoint}
Let $d \ge 3$. 
Assume that the measurable potential $V$ is a real-valued function on $\Om$ 
and satisfies 
$V = V_+ - V_-$, $V_{\pm} \ge 0$ such that 
$V_{\pm} \in K_{d}(\Om)$ and
\begin{equation}\label{EQ:katonorm1}
	\|V_-\|_{K_{d}(\Om)} < \frac{4\pi^{d/2}}{\Gamma(d/2 - 1)}. 
\end{equation}
Let $H_V$ be the operator with domain 
\[
\mathcal{D}(H_V) = \big\{ u \in H^1_0(\Om)\, \big| \, H_V u \in L^2(\Om) \big\},
\]
so that 
\begin{equation}\label{EQ:Schrodinger}
	\big{\langle H_V u , v \big\rangle}_{L^2(\Om)} = 
	\int_{\Om} \nabla u(x)\cdot \overline{\nabla v(x)} \, dx + \int_{\Om} V(x)u(x)\overline{v(x)} \, dx 
\end{equation}
for any $u \in \mathcal{D}(H_V)$ and $v \in H^1_0(\Om)$. 
Then $H_V$ is non-negative and self-adjoint on $L^2(\Om)$. 
\end{prop}

We need a notion of quadratic forms on 
Hilbert spaces (see p.276 in Reed and Simon \cite{RS1}).

\begin{defn}\label{Def:Quadratic}
Let $\Hscr$ be a Hilbert space with the norm $\|\cdot\|$. 
A quadratic form is a map $q : \Qcal(q)\times\Qcal(q)\rightarrow\C$, where $\Qcal(q)$ is a dense linear subset in $\Hscr$ called the form domain, 
such that $q(\cdot,v)$ is conjugate linear and $q(u,\cdot)$ is linear for $u,v\in\Qcal(q)$. 
We say that $q$ is symmetric if $q(u,v)=\overline{q(v,u)}$. 
A symmetric quadratic form $q$ is non-negative if $q(u,u)\ge0$ for any $u\in\Qcal(q)$. 
A non-negative quadratic form $q$ is closed if 
$\Qcal(q)$ is complete with respect to the norm:
\begin{equation}\label{EQ:1-norm}
\|u\|_{+1} := \sqrt{q(u,u)+\|u\|^2}. 
\end{equation}
\end{defn}

\vspace{5mm}

The proof of Proposition \ref{Prop:Self-adjoint} is done by using the following two lemmas. 

\begin{lem}\label{Lem:Self-adjoint}
Let $\Hscr$ be a Hilbert space with the inner product $\langle \cdot,\cdot\rangle$, and let 
$q : \Qcal(q)\times \Qcal(q)\rightarrow \C$ be a densely defined semi-bounded closed quadratic form. 
Then there exists a self-adjoint operator $T$ on $\Hscr$ uniquely such that 
\begin{equation*}
\begin{cases}
	\Dcal(T)=\left\{u\in\Qcal(q)\,|\, \exists \, w_u \in \Hscr \text{ such that } q(u,v)=\langle w_u,v\rangle,\  \forall v\in\Qcal(q)\right\}, \\
	Tu=w_u,\quad u\in\Dcal(T).
\end{cases}
\end{equation*}
\end{lem}
We note that $\Dcal(T)$ can be simply written as
\[
\Dcal(T)=\left\{u\in\Qcal(q)\,|\, Tu\in \Hscr\right\}.
\]
For the proof of Lemma \ref{Lem:Self-adjoint}, see \cite[Theorem VIII.15]{RS1}. \\

The following lemma states that $V_{\pm}$ are relatively form bounded with respect to the 
Dirichlet Laplacian $-\De$. 

\begin{lem}\label{Lem:V}
Let $V_+$ and $V_-$ be as in 
Proposition \ref{Prop:Self-adjoint}. 
Then for any $\ep>0$, there exists a 
constant $b_{\ep}>0$ such that 
the following estimates hold{\rm :} 
\begin{equation}\label{EQ:V+}
	\int_{\Om} V_+ (x) |u(x)|^2 \, dx 
	\le \ep \|\nabla u\|^2_{L^2(\Om)}
+ b_{\ep} \|u\|^2_{L^2(\Om)},
\end{equation}
\begin{equation}\label{EQ:V-}
	\int_{\Om} V_- (x) |u(x)|^2 \, dx 
\le \frac{\|V_-\|_{K_{d}(\Om)}\Gamma(d/2-1)}{4\pi^{d/2}} \|\nabla u\|^2_{L^2(\Om)}
\end{equation}
for any $u \in H^1_0(\Om)$.
\end{lem}

\begin{proof} The proof is similar to 
that of Lemma 3.1 from \cite{DP}.
Let $u \in C^{\infty}_0(\Om)$, and let 
$\tilde{u}$ and $\tilde{V}_{\pm}$ be the 
zero extensions of $u$ and $V_{\pm}$ to 
$\R^d$, respectively. 
First, we prove that for any $\ep > 0$, 
there exists a constant $b_{\ep} > 0$ such that
\begin{equation}\label{EQ:DAncona}
\int_{\R^d} \tilde{V}_+ (x)|\tilde{u}(x)|^2 \, dx 
	\le \ep \|\nabla \tilde{u}\|^2_{L^2(\R^d)}
	+ b_{\ep}\|\tilde{u}\|^2_{L^2(\R^d)}. 
\end{equation}
We divide the proof of \eqref{EQ:DAncona} 
into two cases: $d=3$ and $d>3$.
When $d=3$, 
the inequality 
\eqref{EQ:DAncona} is equivalent to 
\begin{align*}
\int_{\R^3} \tilde{V}_+(x)|\tilde{u}(x)|^2 \, dx 
\le & \ep \langle \tilde{u},-\De\tilde{u}\rangle_{L^2(\R^3)}+b_{\ep}
\|\tilde{u}\|^2_{L^2(\R^3)}\\
=& \ep\left\|\left(H_0+\frac{b_{\ep}}{\ep}\right)^{1/2}\tilde{u}\right\|^2_{L^2(\R^3)},
\end{align*}
where $H_0=-\De$ is the self-adjoint operator with domain $H^2(\R^3)$. 
Put
\[
v = \left(H_0+
\frac{b_{\ep}}{\ep}\right)^{1/2}\tilde{u}. 
\]
Then 
the estimate \eqref{EQ:DAncona} takes the following form: 
\[
\left\|\tilde{V}^{1/2}_+\left(H_0+\frac{b_{\ep}}{\ep}\right)^{-1/2}v\right\|^2_{L^2(\R^3)} \le \ep\|v\|^2_{L^2(\R^3)}. 
\]
This estimate can be obtained if we show that 
\begin{equation}\label{EQ:DAncona2}
\|TT^*\|_{\Bscr(L^2(\R^3))} \le \ep,
\end{equation}
where we set 
\[
T:=\tilde{V}^{1/2}_
+\left(H_0+\frac{b_{\ep}}{\ep}\right)^{-1/2}.
\] 
Thus, it suffices to show that for any $\ep>0$, 
there exists a constant $b_{\ep} > 0$ such that the estimate \eqref{EQ:DAncona2} holds. 
Let $\ep >0$ be fixed and $b>0$. 
By using the formula:
\[
\left(H_0+\frac{b}{\ep}\right)^{-1}v(x) = \frac{1}{4\pi}\int_{\R^3}
\frac{e^{-\sqrt{\frac{b}{\ep}}|x-y|}}{|x-y|}v(y)\,dy
\]
and Schwarz inequality, 
we estimate 
\begin{align*}
&\|TT^*v\|^2_{L^2(\R^3)}\\
=&\left\|\tilde{V}^{1/2}_+\left(H_0
+\frac{b}{\ep}\right)^{-1}\tilde{V}^{1/2}_+v\right\|^2_{L^2(\R^3)}\\
=&\frac{1}{(4\pi)^2}\int_{\R^3}\tilde{V}_+(x)\left|\int_{\R^3}\frac{e^{-\sqrt{\frac{b}{\ep}}|x-y|}}{|x-y|}\tilde{V}^{1/2}_+(y)v(y)\,dy\right|^2\,dx\\
\le&\frac{1}{(4\pi)^2}\int_{\R^3}\tilde{V}_+(x)
\left(\int_{\R^3}\frac{e^{-\sqrt{\frac{b}{\ep}}|x-y|}}{|x-y|}\tilde{V}_+(y)\,dy\right)
\left(\int_{\R^3}\frac{e^{-\sqrt{\frac{b}{\ep}}|x-y|}}{|x-y|}|v(y)|^2\,dy\right)\,dx.
\end{align*}
Now we estimate the first integral on the right. 
We split the integral as follows:
\begin{align*}
\int_{\R^3}\frac{e^{-\sqrt{\frac{b}{\ep}}|x-y|}}{|x-y|}\tilde{V}_+(y)\,dy&=\int_{|x-y|<r}+\int_{|x-y|\ge r}\\
&=:I_1+I_2
\end{align*}
for any $r>0$. 
Let $\de>0$ be fixed. 
Then, if we choose $r>0$ small enough, we have $I_1\le\de$, since $V_+\in K_d(\Om)$. 
Then, choosing $b=b_\de>0$ large enough, we have $I_2\le\de$. 
Thus we obtain 
\begin{equation}\label{EQ:integralV+}
\int_{\R^3}\frac{e^{-\sqrt{\frac{b_\de}{\ep}}|x-y|}}{|x-y|}\tilde{V}_+(y)\,dy\le 2\de.
\end{equation}
Using this estimate, we can estimate 
\[
\|TT^*v\|^2_{L^2(\R^3)}\le \frac{2\de}{(4\pi)^2}\int_{\R^3}\tilde{V}_+(x)
\left(\int_{\R^3}\frac{e^{-\sqrt{\frac{b_\de}{\ep}}|x-y|}}{|x-y|}|v(y)|^2\,dy\right)\,dx. 
\]
Moreover, by using Fubini-Tonelli theorem and the inequality \eqref{EQ:integralV+} once more, 
we estimate
\begin{align*}
\|TT^*v\|^2_{L^2(\R^3)}\le& \frac{2\de}{(4\pi)^2}\int_{\R^3}
\left(\int_{\R^3}\frac{e^{-\sqrt{\frac{b_\de}{\ep}}|x-y|}}{|x-y|}\tilde{V}_+(x)\,dx\right)|v(y)|^2\,dy\\
\le&\left(\frac{2\de}{4\pi}\right)^2\|v\|^2_{L^2(\R^3)}.
\end{align*}
Thus, by choosing $\delta=2\pi \ep$, 
we get \eqref{EQ:DAncona2}, which implies \eqref{EQ:DAncona} for $d=3$. 

When $d>3$, we can also prove the estimate \eqref{EQ:DAncona} in the same argument as in the case when $d=3$, 
if we note that 
the kernel $K_M(x)$ of $(-\De+M)^{-1}$ for $M>0$ satisfies 
\[
|K_M(x)|\le \frac{\Gamma(d/2-1)}{4\pi^{d/2}}\frac{1}{|x|^{d-2}}
\quad \text{and}\quad \lim_{M\rightarrow+\infty}
\sup_{|x|>r}e^{|x|}K_M(x)=0
\]
for each $r>0$ 
(see \cite[p.454]{Si}). 
Indeed, we can perform the argument involving $H_0+\frac{b_{\ep}}{\ep}$ by using the previous 
asymptotics, and as a result, we get also \eqref{EQ:DAncona}. 

Based on \eqref{EQ:DAncona}, 
we can prove 
the required estimates 
\eqref{EQ:V+}. 
In fact, we estimate, by 
using \eqref{EQ:DAncona},
\begin{align*}
	\int_{\Om} V_+ (x) |u(x)|^2 \, dx& = \int_{\R^d} \tilde{V}_+ (x) |\tilde{u}(x)|^2 \, dx \\
	& \le \ep \|\nabla\tilde{u}\|^2_{L^2(\R^d)}
	+ b_{\ep}\|\tilde{u}\|^2_{L^2(\R^d)} \\
	& = \ep \|\nabla u\|^2_{L^2(\Om)}
	+ b_{\ep}\|u\|^2_{L^2(\Om)}. 
\end{align*}
As a consequence, by density argument, the inequality \eqref{EQ:V+} is proved. 

The proof of \eqref{EQ:V-} is almost 
identical to that of \eqref{EQ:V+} by 
regarding $b_\ep$ as $0$. 
The only difference is the estimate \eqref{EQ:integralV+}. 
Instead of \eqref{EQ:integralV+}, we 
can apply the following estimate:
\begin{align*}
\int_{\R^3}\frac{\tilde{V}_-(y)}{|x-y|}\,dy
=&\int_{\Omega}\frac{V_-(y)}{|x-y|}\,dy\\
\le&\|V_-\|_{K_3(\Omega)},
\end{align*}
whence the argument in \eqref{EQ:V+} works well
in this case, and we get \eqref{EQ:V-}. 
The proof of Lemma \ref{Lem:V} is complete. 
\end{proof}

We are now in a position to prove Proposition \ref{Prop:Self-adjoint}. 

\begin{proof}[Proof of Proposition \ref{Prop:Self-adjoint}]
Let $q : H^1_0(\Om)\times H^1_0(\Om)\rightarrow \C$ be the quadratic form by letting 
\[
	q(u,v)=\int_{\Om} (\nabla u \cdot \overline{\nabla v} + Vu\overline{v})\, dx, \quad u,v\in H^1_0(\Om). 
\]
It is clear that $q$ is densely defined and semi-bounded. 
Hence, as a consequence of 
Lemma \ref{Lem:Self-adjoint}, 
it suffices to show that the quadratic form 
$q$ is closed. Hence all we have to do is to 
show that the norm $\|\cdot\|_{+1}$ 
is equivalent to 
that of $H^1_0(\Om)$, where 
$\|\cdot\|_{+1}$ is defined in \eqref{EQ:1-norm},
i.e., 
\[
\|u\|_{+1}=\sqrt{q(u,u)+\|u\|^2_{L^2(\Omega)}}.
\]
In fact, by using Lemma \ref{Lem:V} 
we have 
\begin{align*}
\|u\|^2_{+1} &\le \|\nabla u\|^2_{L^2(\Omega)} 
+ \int_{\Om} V(x) |u(x)|^2 \, dx
+\|u\|^2_{L^2(\Omega)}\\
&\le C\left(\|\nabla u\|_{L^2(\Omega)}
+\|u\|^2_{L^2(\Omega)}\right), 
\end{align*}
and by using assumption \eqref{EQ:katonorm1} on $V$, we see that
\begin{align*}
\|u\|^2_{+1} &\ge \|\nabla u\|^2_{L^2(\Omega)} 
-\int_{\Om} V_-(x) |u(x)|^2 \, dx 
+\|u\|^2_{L^2(\Omega)} \\
&\ge \left(1-\frac{\|V_-\|_{K_d(\Omega)}\Gamma(d/2-1)}{4\pi^{d/2}}\right)\|\nabla u\|^2_{L^2(\Omega)} 
+ \|u\|^2_{L^2(\Omega)}
\end{align*}
for any $u \in H^1_0(\Om)$, 
which implies that $\|\cdot\|_{+1}$ 
is equivalent to 
$\|\cdot\|_{H^1(\Omega)}$. 
The proof of Proposition \ref{Prop:Self-adjoint} is complete.
\end{proof}


\section{$L^p$-$L^q$-estimates and pointwise estimates for $e^{-tH_V}$}
In this section we shall prove 
$L^p$-$L^q$-estimates for $e^{-tH_V}$ and 
pointwise estimates for the integral kernel 
of $e^{-tH_V}$ on $\Om$. 
Throughout this section we use the following 
notation 
\[
\gamma_d=\frac{\pi^{d/2}}{\Gamma(d/2-1)}.
\]
\vspace{0.5cm}

We have the following: 

\begin{prop}\label{Prop:e-tHV}
Assume that the measurable potential $V=V_+-V_-$ satisfies $V_\pm\in K_d(\Om)$. 
Let $1 \le p \le q \le \infty$. Suppose that 
\begin{equation}\label{EQ:katonorm2}
\|V_-\|_{K_{d}(\Om)} < 2 \gamma_d.
\end{equation}
Then 
\begin{equation}\label{EQ:semigroup-LpLq}
	\| e^{-tH_V} f \|_{L^q(\Om)} \le \frac{(2 \pi t)^{-(d/2)(1/p -1/q)}}{\{1 - \|V_-\|_{K_{d}(\Om)}/2 \gamma_d\}^2} \|f\|_{L^p(\Om)}, \ \ \ \forall t > 0
\end{equation}
for any $f \in L^p(\Om)$. 
In addition, if we further assume that $V_-$ satisfies assumption \eqref{EQ:A-2}, i.e., 
\[
	\|V_-\|_{K_{d}(\Om)} < \gamma_d, 
\]
then the kernel $K(t, x, y)$ of $e^{-tH_V}$ enjoys the property that
\begin{equation}\label{EQ:kernel-pointwise}
	0 \le K(t, x, y) \le \frac{(2 \pi t)^{-d/2}}{1 - \|V_-\|_{K_{d}(\Om)}/\gamma_d} 
	e^{-\frac{|x-y|^2}{8t}}, \ \ \ \forall t > 0
\end{equation}
for any $x, y \in \Om$. 
\end{prop}

The following lemma is crucial in the proof of 
Proposition \ref{Prop:e-tHV}.

\begin{lem}\label{Lem:comparison}
Let $d\ge 3$. 
Assume that the measurable potential $V=V_+-V_-$ satisfies 
$V_\pm\in K_d(\Om)$ and 
\begin{equation}\label{EQ:katonorm3}
\|V_-\|_{K_{d}(\Om)} < 4 \gamma_d.
\end{equation} 
Let $\tilde{V}$ be the zero extension of $V$ to $\R^d$ 
and $\tilde{H}_{\tilde{V}}$ the self-adjoint extension of $H_{\tilde{V}}$ on $L^2(\R^d)$. 
Then for any non-negative function $f\in L^2(\Om)$, 
the following estimates hold{\rm :} 
\begin{equation}\label{EQ:positivity}
	\big(e^{-tH_V} f\big)(x) \ge 0, 
\end{equation}
\begin{equation}\label{EQ:comparison}
	\big(e^{-tH_V} f\big)(x) \le \big(e^{-t \tilde{H}_{\tilde{V}}} \tilde{f}\,\big)(x)
\end{equation}
for $t > 0$ and almost everywhere $x \in \Om$, 
where $\tilde{f}$ is the zero extension of $f$ to $\R^d$. 
\end{lem}
The proof of Lemma \ref{Lem:comparison} is rather long, and will be postponed. 
Let us prove Proposition \ref{Prop:e-tHV}.  

\begin{proof}[Proof of Proposition \ref{Prop:e-tHV}]
Let $f \in C^\infty_0(\Om)$. 
Applying \eqref{EQ:positivity} from Lemma \ref{Lem:comparison} to non-negative functions $|f| - f$ and $|f| + f$, we obtain 
\[
	 -\big(e^{-tH_V} |f|\big)(x)\le\big(e^{-tH_V} f\big)(x)\le \big(e^{-tH_V} |f|\big)(x)
\]
for any $t > 0$ and almost everywhere $x \in \Om$. 
Hence the above inequality and \eqref{EQ:comparison} from Lemma \ref{Lem:comparison} imply that 
\begin{equation}\label{EQ:comparison'}
	\big| \big(e^{-tH_V} f\big)(x) \big| \le \big(e^{-t \tilde{H}_{\tilde{V}}} \tilde{|f|}\big)(x)
\end{equation}
for any $t > 0$ and almost everywhere 
$x \in \Om$. 
Here we recall the result of 
$L^p$-$L^q$-estimates for 
$e^{-t \tilde{H}_{\tilde{V}}}$ on $\mathbb{R}^d$:
\begin{equation}\label{EQ:D'Ancona-Pierfelice}
\| e^{-t \tilde{H}_{\tilde{V}}} 
\tilde{|f|}\|_{L^q(\R^d)} 
\le \frac{(2 \pi t)^{-(d/2)(1/p -1/q)}}{\{1-
\|\tilde{V}_-\|_{K_{d}(\R^d)}/2 \gamma_d\}^2} 
\|\tilde{f}\|_{L^p(\R^d)},\quad\forall t>0,
\end{equation}
provided $1 \le p \le q \le \infty$. 
(see Proposition 5.1 from \cite{DP}). 
Combining \eqref{EQ:comparison'} and \eqref{EQ:D'Ancona-Pierfelice}, 
we obtain the estimate \eqref{EQ:semigroup-LpLq} for $f\in C^\infty_0(\Om)$. 
Thus, by density argument, we conclude 
the estimates \eqref{EQ:semigroup-LpLq} 
for any $f\in L^p(\Om)$ if $p<\infty$. 
The case $p=\infty$ follows from the 
duality argument. 

We now turn to prove \eqref{EQ:kernel-pointwise}. 
We adopt a sequence $\{j_{\ep}(x)\}_{\ep>0}$ 
of functions defined by letting
\begin{equation}\label{mollifier}
	j_{\ep}(x) := \frac{1}{\ep^d}\, j\bigg(\frac{x}{\ep}\bigg), \quad x \in \R^d,
\end{equation}
where 
\begin{equation*}
j(x) = 
\begin{cases}
	C_d\, e^{-1/(1-|x|^2)}, \quad &|x| < 1,\\
	0, &|x| \ge 1
\end{cases}
\end{equation*}
with 
\[
	C_d := \bigg(\int_{\R^d} 
	e^{-\frac{1}{1-|x|^2}}\, dx\bigg)^{-1}.
\]
As is well-known, the sequence 
$\{j_{\ep}(x)\}_{\ep}$ enjoys the 
following property: 
\begin{equation}\label{EQ:delta}
	j_{\ep}(\cdot -y) \rightarrow \delta_{y} \quad \text{in } \Sscr'(\R^d) \quad(\ep \rightarrow 0),
\end{equation}
where $\delta_{y}$ is the Dirac delta function at $y \in \Om$. 
Let $y\in\Om$ be fixed, and let $K(t,x,y)$ and $\tilde{K}(t,x,y)$ be kernels of $e^{-tH_V}$ and $e^{-t \tilde{H}_{\tilde{V}}}$, respectively. 
Taking $\ep>0$ sufficiently small so that $\mathrm{supp}\,j_{\ep}(\cdot-y)\Subset\Om$, and 
applying \eqref{EQ:positivity} and \eqref{EQ:comparison} from Lemma \ref{Lem:comparison} to both $f$ and $\tilde{f}$ replaced by $j_{\ep}(\cdot -y)$, we get
\[
	0 \le \int_{\Om} K(t,x,z) j_{\ep}(z-y) dz \le \int_{\R^d} \tilde{K}(t,x,y) j_{\ep}(z-y) dz
\]
for any $x\in \Om$. 
Noting \eqref{EQ:delta} and taking the limit of the previous inequality as $\ep \to 0$, we get
\[
	0 \le K(t,x,y) \le \tilde{K}(t,x,y) 
\]
for any $t > 0$ and $x\in \Om$. 
Finally, by using the pointwise estimates: 
\[
\tilde{K}(t,x,y) \le \frac{(2 \pi t)^{-d/2}}{1 - \|\tilde{V}_-\|_{K_{d}(\R^d)}/\gamma_d} 
e^{-\frac{|x-y|^2}{8t}}
\left(=\frac{(2 \pi t)^{-d/2}}{1 - \|V_-\|_{K_{d}(\Om)}/\gamma_d} e^{-\frac{|x-y|^2}{8t}}\right)
\]
(see Proposition 5.1 from \cite{DP}),
we obtain the estimate \eqref{EQ:kernel-pointwise}, as desired. 
The proof of Proposition \ref{Prop:e-tHV} is finished. 
\end{proof}

In the rest of this section we shall prove Lemma \ref{Lem:comparison}. 
To prove Lemma \ref{Lem:comparison} we need further the following two lemmas. 
The first one is concerned with the existence and uniqueness of solutions for evolution equations in abstract setting. 

\begin{lem}\label{Lem:Semigroup}
Let $\mathscr{H}$ be a Hilbert space. 
Assume that $A$ is a non-negative self-adjoint operator on $\mathscr{H}$. 
Let $\{T(t)\}_{t \ge 0}$ be the semigroup generated by $A$, and let $f \in \mathscr{H}$ and $u(t) = T(t) f$. 
Then $u$ is the unique solution of the following problem{\rm :}
\begin{equation*}
\begin{cases}
	u \in C([0, \infty); \mathscr{H}) 
	\cap C^1((0, \infty); \mathscr{H}), \\
	u(t)\in \mathcal{D}(A), \quad t>0,\\
	u'(t) + Au(t) = 0,\quad t > 0, \\
	u(0) = f.
\end{cases}
\end{equation*}
\end{lem}
For the proof of Lemma \ref{Lem:Semigroup}, 
see, e.g., Cazenave and Haraux 
\cite[Theorem 3.2.1]{CH}. \\

The second one is about the differentiability properties for composite functions of Lipschitz continuous functions and $W^{1,p}$-functions. 

\begin{lem}\label{Lem:Miyajima}
Let $\Om$ be an open set in $\R^d$, 
where $d\ge1$, and let $1 \le p \le \infty$. 
Consider the positive and negative parts of a real-valued function $u \in W^{1,p}(\Om)${\rm :}
\[
	u^+ = \chi_{\{u>0\}} u \quad 
	\text{and} \quad 
	u^- = - \chi_{\{u<0\}} u. 
\]
Then $u^{\pm} \in W^{1,p}(\Om)$ and 
\[
	\partial_{x_j} u^{+} = \chi_{\{u>0\}} \partial_{x_j} u, \quad \partial_{x_j} u^{-} = - \chi_{\{u<0\}} \partial_{x_j} u \quad (1 \le j \le d),
\]
where $\displaystyle \pa_{x_j}=\pa/\pa x_j$. 
\end{lem}
\vspace{2mm}
For the proof of Lemma \ref{Lem:Miyajima}, see Gilbarg and Trudinger \cite[Lemma 7.6]{GT}. \\

To prove \eqref{EQ:positivity}, we show that the negative part of $e^{-tH_V}f$ vanishes in $\Om$, provided $f\ge0$. 
For this purpose, we prepare the following lemma.  

\begin{lem}\label{Lem:Negative}
Let $V$ be as in Lemma \ref{Lem:comparison},
and $f \in L^2(\Om)$ a non-negative function. 
Put $$u(t) = e^{-tH_V}  f,\quad t>0.$$
Then the negative part $u^-(t)$ of $u(t)$ belongs to $H^1_0(\Om)$ for each $t>0$. 
\end{lem}

\begin{proof}
Lemma \ref{Lem:Semigroup} assures that 
$$
\left\{
\begin{aligned}
& u \in C([0, \infty); L^2(\Om)) \cap C^1 ((0, \infty); L^2(\Om)),\\
& u(t) \in H^1_0(\Omega),\quad H_Vu(t)\in L^2(\Omega),\quad t>0,\\
& \pa_tu(t)+H_Vu(t)=0,\quad t>0,\\
& u(0)=f.
\end{aligned}
\right.
$$
Since $u(t) \in H^1_0(\Om)$ for each $t > 0$, there exist $\va_n(t) \in C^{\infty}_0(\Om)$ ($n=1,2,\ldots$) such that 
\begin{equation}\label{EQ:H1-density}
	\va_n(t) \rightarrow u(t) \quad \text{in } H^1(\Om)
\end{equation}
as $n \rightarrow \infty$ for each $t>0$. 
Here $\{\va_n\}_n$ also depends on $t$. For the sake of simplicity, 
we may omit the time variable $t$ of $\va_n$ without any confusion.  
Let us take a non-negative function $\psi \in C^{\infty}(\R)$ as 
\begin{equation*}
\psi(x) 
\begin{cases}
	= -x,\quad &x \le -1,\\
	\le -x, &-1 < x < 0, \\
	= 0, &x \ge 0, 
\end{cases}
\end{equation*}
and put 
\begin{equation} \label{EQ:psi-n}
	\psi_n(x) := \frac{\psi(nx)}{n}, \quad n =1,2,\cdots. 
\end{equation}
Then there exists a constant $C > 0$ such that 
\begin{equation}\label{EQ:psi}
	|\psi'_n(x)| \le C, \quad \forall x \in \R, \ \forall n \in \N.
\end{equation}
Let us consider two kinds of 
composite functions $\psi_n \circ \va_n$ and 
$\psi_n \circ u$. 
We show that 
\begin{equation}\label{Lebesgue1-1}
	\psi_n \circ \va_n - \psi_n \circ u \rightarrow 0 \quad \text{in } H^1(\Om), 
\end{equation}
\begin{equation}\label{Lebesgue1-2}
	\psi_n \circ u - u^- \rightarrow 0 \quad \text{in } H^1(\Om)
\end{equation}
as $n \rightarrow \infty$. 
In fact,  
by the mean value theorem, we have 
\begin{align}\label{EQ:L1}
	\|\psi_n \circ \va_n - \psi_n \circ u\|_{L^2(\Om)}=&\bigg\| \int^{1}_{0} \psi'_n \big(\theta \va_n + (1 -\theta)u\big) (\va_n - u) \, d\theta \bigg\|_{L^2(\Om)}\\
	\le& C \|\va_n - u\|_{L^2(\Om)}, \nonumber
\end{align} 
and the derivative of $\psi_n \circ \va_n - \psi_n \circ u$ is written as
\begin{align}\label{EQ:L2}
	&\|\partial_{x_j} (\psi_n \circ \va_n - \psi_n \circ u)\|_{L^2(\Om)}\\
	=& \|\psi'_n(\va_n) \partial_{x_j} \va_n - \psi'_n(u) \partial_{x_j} u \|_{L^2(\Om)} \nonumber \\
	\le&  \|\psi'_n(\va_n) (\partial_{x_j} \va_n - \partial_{x_j} u)\|_{L^2(\Om)} + \|\{\psi'_n(\va_n) - \psi'_n(u)\} \partial_{x_j} u \|_{L^2(\Om)} \nonumber \\
	\le& C \|\partial_{x_j} \va_n - \partial_{x_j} u\|_{L^2(\Om)} + \|\{ \psi'_n(\va_n) - \psi'_n(u)\}\partial_{x_j} u \|_{L^2(\Om)}, \nonumber 
\end{align}
where we used \eqref{EQ:psi} in the last step. 
Noting the pointwise convergence and uniform boundedness with respect to $n$:
\begin{align*}
	&\big\{\psi'_n(\va_n)(x) - \psi'_n(u)(x)\big\} \partial_{x_j} u(x) \rightarrow 0 \quad \text{as $n\to\infty$ for a.e. $x\in\Om$},\\
	&\big|\big\{\psi'_n(\va_n)(x) - \psi'_n(u)(x)\big\}\partial_{x_j} u(x) \big| \le 2C |\partial_{x_j} u(x)| \in L^2(\Om),
\end{align*}
we can apply Lebesgue's dominated convergence theorem 
to obtain 
\begin{equation}\label{EQ:L3}
	\|\{ \psi'_n(\va_n) - \psi'_n(u) \} \partial_{x_j} u \|_{L^2(\Om)} \rightarrow 0 
\end{equation}
as $n \rightarrow \infty$. 
Hence, summarizing \eqref{EQ:H1-density}
and \eqref{EQ:L1}--\eqref{EQ:L3}, we 
obtain \eqref{Lebesgue1-1}. 

As to the latter convergence \eqref{Lebesgue1-2}, since
\[
	|(\psi_n \circ u)(x) - u^- (x)| \le 2 |u(x)| \in L^2(\Om),
\]
\[
	|\partial_{x_j} (\psi_n \circ u)(x) - \partial_{x_j} u^- (x)| \le (C+1) |\partial_{x_j} u(x)| \in L^2(\Om),
\]
and 
\[
	(\psi_n \circ u)(x) - u^- (x) \rightarrow 0, 
\]
\[
	\partial_{x_j} (\psi_n \circ u)(x) - \partial_{x_j} u^- (x) = \{\psi'_n(u) +\chi_{\{u<0\}}\}\partial_{x_j} u (x) \rightarrow 0
\]
as $n \rightarrow \infty$ for almost everywhere $x \in \Om$, 
Lebesgue's dominated convergence theorem allows as to obtain \eqref{Lebesgue1-2}. 
Thus \eqref{Lebesgue1-1} and \eqref{Lebesgue1-2} imply that 
\begin{equation*}
	\psi_n \circ \va_n - u^- \rightarrow 0 \quad \text{in } H^1(\Om) \quad (n \rightarrow \infty).
\end{equation*}
Since 
$\psi_n \circ \va_n \in C^{\infty}_0 (\Om)$, 
we conclude that 
$u^- \in H^1_0(\Om)$.
The proof of Lemma \ref{Lem:Negative} is finished.
\end{proof}

We are now in a position to prove Lemma \ref{Lem:comparison}. 

\begin{proof}[Proof of Lemma \ref{Lem:comparison}.] 
Let $f \in L^2(\Om)$ be non-negative almost everywhere on $\Om$. 
We recall that  
\[
u(t) = e^{-tH_V}f \quad \text{for $t\ge 0$.}
\] 
If we show that $\|u^-(t)\|^2_{L^2(\Om)}$ 
is monotonically decreasing with respect to 
$t\ge 0$, then we obtain 
\[
u^-(t,x) = 0 
\]
for each $t > 0$ and almost everywhere $x \in \Om$, since $u^-(0,x) = f^-(x) = 0$ for almost everywhere $x \in \Om$. 
This means that 
\[
u(t,x) \ge 0
\] 
for each $t > 0$ and almost everywhere $x \in \Om$; thus we conclude \eqref{EQ:positivity}. 
Hence it is sufficient to show that 
\begin{equation}\label{EQ:MD}
\frac{d}{dt} \int_{\Om} \big(u^{-} \big)^2 \,dx\le 0.
\end{equation}
By the definition of $u^+$, we have $\pa_tu^+(t,x)=0$ for $x\in\{u<0\}$ and each $t>0$. 
We compute
\begin{align} \label{EQ:Diff}
	\frac{d}{dt} \int_{\Om} \big(u^{-} \big)^2\,dx=& 2 \int_{\Om} u^{-} \partial_t u^{-}\,dx\\
	\nonumber
	=& 2 \int_{\{u<0\}} u^{-} \partial_t \big(u^+-u\big) \,dx \\
	\nonumber
	=& -2 \int_{\{u<0\}} u^{-}\partial_t u\,dx\\
	\nonumber
	=& 2 \int_{\Omega} (H_V u) u^{-}\,dx
	\nonumber
\end{align}	
where we use the equation $\partial_t u +H_Vu=0$ in the last step. 
Since $u^-\in H^1_0(\Om)$ by Lemma \ref{Lem:Negative}, 
we have, by going back to 
\eqref{EQ:Schrodinger} in 
the definition of $H_V$,
\begin{equation}\label{EQ:aaa}
\int_{\Omega} (H_V u)u^{-}\,dx=  \int_{\Omega} \nabla u \cdot \nabla u^{-}\,dx + \int_{\Omega} V u u^{-}\,dx.
\end{equation}
Here we see from Lemma \ref{Lem:Miyajima} that
\[
\nabla u^-=-\chi_{\{u<0\}}\nabla u,
\]
and hence, 
the first term on the right of \eqref{EQ:aaa}
is written as 
\[
\int_{\Omega} \nabla u \cdot \nabla u^{-}\,dx =- \int_{\Omega} |\nabla u^{-}|^2\,dx.
\]
As to the second, by the estimate \eqref{EQ:V-} from Lemma \ref{Lem:V}, we have
\begin{align*}
\int_{\Omega} V u u^{-}\,dx=& - \int_{\Omega} V |u^{-}|^2 \,dx\\
\le&\int_{\Omega} V_- |u^{-}|^2\,dx \\
\le&\frac{\|V_-\|_{K_d(\Omega)}}{4\gamma_d}\int_{\Omega} |\nabla u^{-}|^2\,dx;
\end{align*}
thus we find from assumption 
\eqref{EQ:katonorm3} on $V$ that 
\begin{align*}
\int_{\Omega} (H_V u)u^{-}\,dx\le& -\left(1-\frac{\|V_-\|_{K_d(\Omega)}}{4\gamma_d}\right)\int_{\Omega} |\nabla u^{-}|^2\,dx\\
\le& 0,
\end{align*}
and hence, combining 
this inequality and \eqref{EQ:Diff},  
we conclude \eqref{EQ:MD}. \\

Next, we prove \eqref{EQ:comparison}. 
Let us define two functions 
$v^{(1)}(t)$ and $v^{(2)}(t)$ as follows:
\[
	v^{(1)}(t) := e^{-t \tilde{H}_{\tilde{V}}} \tilde{f}\quad \text{and} \quad v^{(2)}(t) := e^{-t H_V} f
\]
for $t\ge0$. 
Then it follows from Lemma \ref{Lem:Semigroup} that $v^{(1)}$ and $v^{(2)}$ satisfy
\begin{equation}\label{EQ:v1}
\begin{cases}
	v^{(1)} \in C([0, \infty); L^2(\R^d)) \cap C^1 ((0, \infty); L^2(\R^d)),\\
	v^{(1)}(t) \in H^1(\R^d), \quad \tilde{H}_{\tilde{V}} v^{(1)}(t) \in L^2(\R^d), \\
	\partial_t v^{(1)} (t) + \tilde{H}_{\tilde{V}} v^{(1)}(t)= 0,\\
	v^{(1)}(0) = \tilde{f}
\end{cases}
\end{equation}
and
\begin{equation}\label{EQ:v2}
\begin{cases}
	v^{(2)} \in C([0, \infty); L^2(\Om)) \cap C^1 ((0, \infty); L^2(\Om)),\\
	v^{(2)}(t) \in H^1_0(\Om), \quad H_V v^{(2)}(t) \in L^2(\Om),\\
	\partial_t v^{(2)}(t) + H_V v^{(2)}(t)= 0,\\
	v^{(2)}(0) = f
\end{cases}
\end{equation}
for each $t > 0$, respectively. 
We define a new function $v$ as
\[
v(t) := v^{(1)}(t)|_{\Om} - v^{(2)}(t)
\]
for $t\ge0$, where $v^{(1)}(t)|_{\Om}$ is the restriction of $v^{(1)}(t)$ to $\Om$.
Let us consider the negative part of $v$: 
\[
v^- = - \chi_{\{v<0\}} v.
\]
Then, thanks to 
\eqref{EQ:v1} and \eqref{EQ:v2}, we have 
\[
v^{-} \in C([0, \infty); L^2(\Om)) 
\cap C^1((0, \infty); L^2(\Om)). 
\]
Moreover, by using Lemma \ref{Lem:Miyajima}, 
we have $v^- \in H^1(\Om)$,
since $v \in H^1(\Om)$. 
Once we prove that 
\begin{equation}\label{negative-v}
	v^- \in H^1_0(\Om), 
\end{equation}
we can get, by the previous argument, 
\begin{equation}\label{EQ:MDv-}
\frac{d}{dt} \int_{\Om} \big(v^{-} \big)^2\,dx \le 0.
\end{equation}
In fact, 
by the definition of $v^-$, we have
\begin{align*}
\frac{d}{dt} \int_{\Omega} 
\big(v^{-} \big)^2\,dx=&-2 \int_{\{v<0\}} 
v^{-} \partial_t v^{(1)}\,dx + 2 \int_{\{v<0\}} 
v^{-} \partial_t v^{(2)}\,dx \\
=& 2 \int_{\R^d} (\tilde{H}_{\tilde{V}} v^{(1)}) \tilde{v}^{-}\,dx - 2 
\int_{\Omega} 
(H_V v^{(2)}) v^{-}\,dx,
\end{align*}	
where $\tilde{v}^-$ is the zero extension of 
$v^-$ to $\R^d$, and 
we use equations 
$\partial_t v^{(1)}+\tilde{H}_{\tilde{V}}v^{(1)}=0$ 
and $\partial_t v^{(2)}+H_Vv^{(2)}=0$ in the last step. 
Since $v^-\in H^1_0(\Om)$ by \eqref{negative-v}, 
we have, by definitions of $\tilde{H}_{\tilde{V}}$ and $H_V$,
\begin{align*}
& \int_{\R^d} (\tilde{H}_{\tilde{V}} v^{(1)}) \tilde{v}^{-}\,dx -  
\int_{\Omega} 
(H_V v^{(2)}) v^{-}\,dx\\
=&\int_{\R^d} \nabla v^{(1)} \cdot \nabla \tilde{v}^{-}\,dx 
+ \int_{\R^d} \tilde{V} v^{(1)} \tilde{v}^{-}\,dx
- \int_{\Omega} \nabla v^{(2)} \cdot \nabla v^{-} \,dx
- \int_{\Omega} V v^{(2)}v^{-}\,dx \\
=&\int_{\Omega} \nabla v \cdot \nabla v^{-}\,dx 
+ \int_{\Omega} V v v^{-}\,dx\\
\le& -\left(1-\frac{\|V_-\|_{K_d(\Omega)}}{4\gamma_d}\right)\int_{\Om} |\nabla v^{-}|^2\,dx\\
\le&\ 0,
\end{align*}
where we used assumption A in the last step. 
Hence we obtain \eqref{EQ:MDv-}, which 
implies the required inequality 
\eqref{EQ:comparison}. \\

It remains to prove \eqref{negative-v}. 
The proof is similar to that of Lemma \ref{Lem:Negative}. 
Since $v^{(2)}(t) \in H^1_0(\Om)$ for each $t>0$ by \eqref{EQ:v2}, there exist $\va_n=\va_n(t) \in C^{\infty}_0(\Om)$ such that 
\[
\va_n \rightarrow v^{(2)} \quad \text{in } H^1(\Om)
\]
as $n \rightarrow \infty$. 
Put 
\[
v_n(t) := v^{(1)}(t)|_{\Om} - \va_n(t),
\quad n=1,2,\cdots,
\]
for each $t>0$. 
Let $\psi_n$ be as in \eqref{EQ:psi-n}. 
As in the proof of Lemma \ref{Lem:Negative}, 
we can show that 
\[
\psi_n \circ v^-_n - v^- \rightarrow 0 
\quad \text{in } H^1(\Om) 
\]
as $n \rightarrow \infty$.
Since $v^-_n$ have compact supports in 
$\mathrm{supp}\, \va_n$ 
by $v^{(1)}\ge0$ on $\Omega$, the functions 
$\psi_n \circ v^-_n$ also have compact supports in $\Om$. 
Let $\widetilde{\psi_n \circ v^-_n}$ be the zero extension of $\psi_n \circ v^-_n$ to $\R^d$, 
and let $J_{\ep}$ be Friedrichs' mollifier: For $u \in L^1_{\mathrm{loc}}(\R^d)$, 
\[
	(J_{\ep} u)(x) := (j_{\ep} * u)(x) = \int_{\R^d} j_{\ep}(x-y) u(y) \, dy, \quad x \in \R^d,
\]
where $\{j_\ep(x)\}_{\ep}$ are functions 
defined in \eqref{mollifier}. 
Taking $\ep =\ep_n$ sufficiently small so that 
$\ep_n\to 0$ ($n\to\infty$) and 
$\mathrm{supp}\, J_{\ep_n}\left(
\widetilde{\psi_n \circ v^-_n} \right)$ 
is contained in $\Om$, 
we have 
\[
J_{\ep_n}(\widetilde{\psi_n \circ v^-_n})|_{\Om} \in C^{\infty}_0(\Om).
\] 
Since 
\[
	J_{\ep_n}\left(\widetilde{\psi_n \circ v^-_n} \right)\Big|_{\Om} - v^- \rightarrow 0 \quad \text{in } H^1(\Om) 
\]
as $n \rightarrow \infty$, 
we conclude \eqref{negative-v}. 
The proof of Lemma \ref{Lem:comparison} is complete. 
\end{proof}


\section{$L^p$-$\ell^p(L^q)_{\theta}$-boundedness for the resolvent of $\theta H_V $}

In this section we shall prove the boundedness 
of resolvent $(\theta H_V - z)^{-\beta}$ 
$(\beta>0)$ in scaled amalgam spaces. 
The result in this section plays an 
important role in the proof of Theorem \ref{Thm1.1}. \\

More precisely, we have:

\begin{thm}\label{Thm4.1}
Assume that the measurable potential $V=V_+-V_-$ satisfies $V_\pm\in K_d(\Om)$, and 
that $V_-$ satisfies 
assumption \eqref{EQ:katonorm2} 
in Proposition \ref{Prop:e-tHV}.
Let  $1 \le p \le q \le \infty$ and 
$\beta > (d/2)(1/p - 1/q)$, 
and let $z \in \mathbb{C}$ with ${\rm Re}(z) < 0$. Then there exists a constant $C = C(d, p, q, \beta, z) > 0$ such that 
\begin{equation}\label{Th4.1-1}
	\| (\theta H_V - z)^{- \beta} \|_{\mathscr{B}(L^p(\Om), L^q(\Om))} 
	\le C {\theta}^{- (d/2)(1/p-1/q)}
\end{equation}
for any $\theta > 0$. If we further assume that $V_-$ satisfies assumption \eqref{EQ:A-2},
then 
\begin{equation}\label{Th4.1-2}
	\| (\theta H_V - z)^{- \beta} \|_{\mathscr{B}(L^p(\Om), \ell^p(L^q)_{\theta})} \le C {\theta}^{- (d/2)(1/p-1/q)}
\end{equation}
for any $\theta > 0$. 
\end{thm}

\begin{proof}
Let us first prove \eqref{Th4.1-1}. 
We use the following well-known formula: For $z \in \mathbb{C}$ with ${\rm Re}(z) < 0$ and $\beta > 0$, 
\begin{equation}\label{4-1}
	(H_V - z)^{- \beta} = \frac{1}{\Gamma(\beta)} \int^{\infty}_{0} t^{\beta - 1} e^{z t} e^{-tH_V} \,dt. 
\end{equation}
Since $V$ satisfies assumption \eqref{EQ:katonorm2}, 
thanks to the formula \eqref{4-1} and $L^p$-$L^q$-estimates \eqref{EQ:semigroup-LpLq} for $e^{-t \theta H_V}$ in Proposition \ref{Prop:e-tHV}, 
we can estimate 
\begin{align*}
	 \| (\theta H_V - z)^{- \beta} f \|_{L^q(\Om)}& \le \frac{1}{\Gamma(\beta)} \int^{\infty}_{0} t^{\beta - 1} e^{{\rm Re}(z) t} \| e^{-t \theta H_V} f \|_{L^q(\Om)} \,dt \\
	& \le C\theta^{-(d/2)(1/p-1/q)} \int^{\infty}_{0} t^{\beta - 1} e^{{\rm Re}(z) t} 
	t^{-(d/2)(1/p-1/q)} \,dt \cdot \|f\|_{L^p(\Om)}.
\end{align*}
Since $\beta > (d/2)(1/p - 1/q)$ and ${\rm Re}(z) < 0$, the integral on the right is absolutely 
convergent. 
Hence we obtain
\[
	\| (\theta H_V - z)^{- \beta} f \|_{L^q(\Om)} \le C\theta^{-(d/2)(1/p-1/q)} \|f\|_{L^p(\Om)}.
\]
This proves \eqref{Th4.1-1}. \\

Let us turn to the proof of \eqref{Th4.1-2}. If we can prove that
\begin{equation}\label{4.1-2}
	\|e^{-t \theta H_V} f\|_{\ell^p(L^q)_{\theta}} \le C 
	\theta^{-(d/2)(1/p-1/q)} 
	\left\{t^{-(d/2)(1/p-1/q)} + 1\right\} \|f\|_{L^p(\Om)},\quad \forall t > 0
\end{equation} 
for any $f \in L^p(\Om)$ provided 
$1 \le p \le q \le \infty$,
then the estimate \eqref{Th4.1-2} is obtained by combining \eqref{4-1} and \eqref{4.1-2}. 
In fact, 
we estimate  
\begin{align*}
& \| (\theta H_V - z)^{- \beta} 
f \|_{\ell^p(L^q)_{\theta}}\\
\le & \frac{1}{\Gamma(\beta)} 
\int^{\infty}_{0} t^{\beta - 1} e^{{\rm Re}(z) t} 
\| e^{-t \theta H_V} f 
\|_{\ell^p(L^q)_{\theta}} \,dt \\
\le &C\theta^{-(d/2)(1/p-1/q)} 
\int^{\infty}_{0} t^{\beta - 1} e^{{\rm Re}(z) t} 
\left\{t^{-(d/2)(1/p-1/q)} + 1\right\} \,dt 
\cdot \|f\|_{L^p(\Om)}.
\end{align*}
Since $\beta >(d/2)(1/p-1/q)$ and 
${\rm Re}(z) < 0$, the integral on 
the right is absolutely convergent. 
Hence 
we conclude that 
\[
	\| (\theta H_V - z)^{- \beta} \|_{\ell^p(L^q)_{\theta}} \le C 
	\theta^{-(d/2)(1/p-1/q)} \|f\|_{L^p(\Om)}.
\]
This proves \eqref{Th4.1-2}. 
Therefore, all we have to do is to prove the estimate \eqref{4.1-2}. 
To this end, we prove the following estimate:
For $1 \le q \le \infty$ and any $\theta > 0$,
\begin{equation}\label{4.1-1}
	\| K_0(\theta t, \cdot) \|_{\ell^1(L^q)_{\theta}} \le C 
	\theta^{-(d/2)(1/p-1/q)} 
	\left\{t^{-(d/2)(1/p-1/q)} + 1\right\},\ \ \ \forall t > 0,
\end{equation}
where $K_0(t,x)$ is defined by letting
\[
	K_0(t, x) = \frac{(2 \pi t)^{-d/2}}{1 - \|V_-\|_{K_{d}(\Om)}/\gamma_d} e^{-\frac{|x|^2}{8t}}=:C_1t^{-d/2}e^{-\frac{|x|^2}{8t}}
\]
for any $t > 0$ and $x \in \R^d$. 
Here, recalling that 
\[
\gamma_d = \frac{\pi^{d/2}}{\Gamma(d/2-1)},  
\]
we note from assumption \eqref{EQ:A-2} on $V$ that 
\[
C_1=\frac{(2 \pi)^{-d/2}}{1 - \|V_-\|_{K_{d}(\Om)}/\gamma_d}>0.
\]
For the proof of \eqref{4.1-1}, 
we compute $\| K_0(\theta t, \cdot) \|_{L^q(C_{\theta}(n))}$ for the case $n = 0$ and $n \neq 0$, separately: \\

{\bf The case $n=0$}: We estimate
\begin{align}\label{EQ:n=0}
	 \| K_0(\theta t, \cdot) 
	 \|_{L^q(C_{\theta}(0))}\le& 
	C_1(\theta t)^{-d/2} \bigg( \int_{|x| < \sqrt{d\theta}} e^{-\frac{ q|x|^2}{8\theta t}} \,dx \bigg)^{1/q} \\
	=& C_1(\theta t)^{-d/2} \bigg( \int_{|x| < \sqrt{\frac{d}{t}}} e^{- \frac{q |x|^2}{8}} (\theta t)^{d/2} \,dx \bigg)^{1/q}\nonumber \\
	\le & C (\theta t)^{-(d/2)(1 - 1/q)} \bigg( \int_{\R^d} e^{- \frac{q |x|^2}{8}} \,dx \bigg)^{1/q}\nonumber \\
	 \le & C (\theta t)^{-(d/2)(1 - 1/q)}.\nonumber
\end{align}
{\bf The case $n\ne0$}: We estimate
\begin{align}\label{EQ:n-neq-0}
	\sum_{n \neq 0} \| K_0(\theta t, \cdot) \|_{L^q(C_{\theta}(n))}& =
	C_1(\theta t)^{-d/2} \sum_{n \neq 0} \bigg( \int_{C_{\theta}(n)} e^{- \frac{q |x|^2}{8\theta t}} \,dx \bigg)^{1/q} \\
& \le C_1(\theta t)^{-d/2} \sum_{n \neq 0}\bigg(  \sup_{x \in C_{\theta}(n)} 
e^{- \frac{|x|^2}{8\theta t}}\bigg) \cdot\bigg(\int_{C_{\theta}(n)} \,dx \bigg)^{1/q}.\nonumber
\end{align}
Here, observing that 
\[
\frac{|\theta^{1/2}n|}{2}\le |x| \, (\le 2|\theta^{1/2}n|), \quad x \in C_{\theta}(n),
\]
we can estimate the right member of \eqref{EQ:n-neq-0} as 
\[
 C_1(\theta t)^{-d/2} \bigg( \sum_{n \neq 0} e^{- \frac{|n|^2}{32t}} \bigg) (\theta^{d/2})^{1/q},
\]
and hence, we get 
\begin{equation*}
\sum_{n \neq 0} \| K_0(\theta t, \cdot) \|_{L^q(C_{\theta}(n))}\le C_1(\theta t)^{-d/2} \bigg( \sum_{n \neq 0} e^{- \frac{|n|^2}{32t}} \bigg) (\theta^{d/2})^{1/q}.
\end{equation*}
Here, by an explicit calculation, we see that 
\begin{align*}
\sum_{n \neq 0} e^{- \frac{|n|^2}{32t}} 
=& \sum_{n \neq 0} e^{- \frac{n^2_1+n^2_2+\cdots+n^2_d}{32t}}
=2^d\left(
\sum^{\infty}_{j=1}e^{- \frac{j^2}{32t}}\right)^d\\
\le&2^d\left(\int^{\infty}_{0}e^{- \frac{\sigma^2}{32t}}\,d\sigma\right)^d\nonumber\\
=&(8\sqrt{2})^d\pi^{d/2} t^{d/2}\nonumber.
\end{align*}
Summarizing the estimates obtained now, we conclude that 
\begin{align}\label{EQ:large}
\sum_{n \neq 0} \| K_0(\theta t, \cdot) \|_{L^q(C_{\theta}(n))}
\le& C_1(\theta t)^{-d/2} \cdot (8\sqrt{2})^d\pi^{d/2} t^{d/2}\cdot (\theta^{d/2})^{1/q}\\
=&(8\sqrt{2})^d\pi^{d/2}C_1
\theta^{-(d/2)(1/p-1/q)}.\nonumber
\end{align}
Combining the estimates \eqref{EQ:n=0}--\eqref{EQ:large}, we obtain \eqref{4.1-1}, as desired.

We are now in a position to prove the key estimate \eqref{4.1-2}. 
Let $f \in L^p(\Om)$ and $\tilde{f}$ be a zero extension of $f$ to $\R^d$. 
Thanks to the estimate \eqref{EQ:kernel-pointwise} in Proposition \ref{Prop:e-tHV}, we have 
\begin{align*}
	\|e^{-t \theta H_V} f\|_{\ell^p(L^q)_{\theta}} =& \left\|\int_{\Om}K(\theta t,\cdot,y)f(y)\,dy\right\|_{\ell^p(L^q)_{\theta}} \\
	\le& \left\|\int_{\Om}K(\theta t,\cdot,y)|f(y)|\,dy\right\|_{\ell^p(L^q)_{\theta}} \\
	\le& \left\|\int_{\R^d}K_0(\theta t,\cdot-y)|\tilde{f}(y)|\,dy\right\|_{\ell^p(L^q)_{\theta}(\R^d)}. 
\end{align*}
Applying the Young inequality \eqref{EQ:Young} (see appendix \ref{App:AppendixA}) to the right member, and using the inequality \eqref{4.1-1}, we estimate
\begin{align*}
	\|e^{-t \theta H_V} f\|_{\ell^p(L^q)_{\theta}}& \le 3^d 
	\| K_0(\theta t, \cdot) \|_{\ell^1(L^r)_{\theta}(\R^d)} \|\tilde{f}\|_{\ell^p(L^p)_{\theta}(\R^d)} \\
	& \le C \theta^{- (d/2)(1-1/r)} 
	\left\{t^{- (d/2)(1-1/r)} + 1\right\} \|\tilde{f}\|_{L^p(\R^d)} \\
	& = C \theta^{-(d/2)(1/p-1/q)} 
	\left\{t^{-(d/2)(1/p-1/q)}+ 1\right\} 
	\|f\|_{L^p(\Om)},
\end{align*}
provided that $p,q,r$ satisfy $1 \le p,q,r \le \infty$ and $1/p + 1/r - 1 = 1/q$. 
This proves \eqref{4.1-2}. 
The proof of Theorem \ref{Thm4.1} is finished. 
\end{proof}


\section{Commutator estimates}
In this section we shall prepare commutator estimates. These estimates will be also 
an important tool in the proof of Theorem 
\ref{Thm1.1}. 
Among other things, we introduce an operator $\mathrm{Ad}$  as follows:\\

\noindent{\bf Definition.}\ {\it 
Let $X$ and $Y$ be topological vector spaces, 
and let $A$ and $B$ be continuous 
linear operators from $X$ and $Y$ into 
themselves, respectively. 
For a continuous linear operator $L$ from 
$X$ into $Y$, the operator $\mathrm{Ad}^k(L)$ from 
$X$ into $Y$, $k = 0, 1, \cdots$, is 
successively defined by
\[ 
\mathrm{Ad}^0(L) = L,\quad 
\mathrm{Ad}^k(L) = \mathrm{Ad}^{k-1}(BL-LA),\quad k \ge 1.
\]
}

The result in this section is concerned with $L^2$-boundedness for 
$\mathrm{Ad}^k(e^{-it R_{V,\theta}})$, 
where $R_{V,\theta}$ is the resolvent operator defined by letting
\[
	R_{V,\theta} := (\theta H_V + M)^{-1},\quad \theta > 0
\]
for a fixed $M > 0$. 
Hereafter, operators $A$ and $B$ are taken as  
\begin{equation}\label{EQ:AB}
	A=B=x_j - \theta^{1/2} n_j\quad  
\text{for some $j\in\{1,\cdots,d\}.$}
\end{equation}
Then we shall prove here the following. 

\begin{prop}\label{prop:Prop5.1}
Let $d \ge 3$. 
Assume that the measurable potential 
$V = V_+ - V_-$ satisfies  
$V_{\pm} \in K_{d}(\Om)$, and that 
$V_-$ satisfies assumption \eqref{EQ:katonorm3} in Lemma \ref{Lem:comparison}.
Let $A$ and $B$ be the operators as 
in \eqref{EQ:AB}, and let $L = e^{-it R_{V,\theta}}$.
Then for any non-negative integer $k$, 
there exists a constant $C = C (d,M,k)>0$ 
such that 
\begin{equation}\label{EQ:Prop5.1}
	\|\mathrm{Ad}^k(e^{-itR_{V,\theta}})\|_{\mathscr{B}(L^2(\Om))} \le C \theta^{k/2}  (1+t)^k
\end{equation}
for any $t>0$ and $\theta > 0$.
\end{prop}

First, we prepare $L^2$-boundedness 
for $R_{V,\theta}$ and 
$\partial_{x_j} R_{V,\theta}$ to prove 
Proposition \ref{prop:Prop5.1}. 

\begin{lem}\label{Lem:Lem5.2}
Let $d \ge 3$ and $V$ be as in Proposition \ref{prop:Prop5.1}. 
Then the following estimates hold{\rm :} 
\begin{equation} \label{EQ:Lem5.2-1}
	\| R_{V,\theta} \|_{\mathscr{B}(L^2(\Om))} \le M^{-1}, 
	\end{equation}
	\begin{equation}\label{EQ:Lem5.2-2}
	\| \nabla R_{V,\theta} \|_{\mathscr{B}(L^2(\Om))} 
	\le M^{-1/2}\left(1 - \frac{\|V_-\|_{K_{d}(\Om)}\Gamma(d/2-1)}{4\pi^{d/2}}\right)^{-1/2} \theta^{-1/2}
\end{equation}
for any $\theta > 0$. 
\end{lem}

\begin{proof}
Since $H_V$ is the self-adjoint operator with domain 
\[
\mathcal{D}(H_V) = \{u \in H^1_0 (\Om) \ | \ H_V u \in L^2(\Om)\},
\] 
we obtain \eqref{EQ:Lem5.2-1} and \eqref{EQ:Lem5.2-2} by the spectral resolution. 
In fact, we have 
\begin{align*}
\|R_{V,\theta}f\|^2_{L^2(\Om)}
=&\int^{\infty}_0 \frac{1}{(\theta \la + M )^2}\, 
d\|E_{H_V}(\la) f\|^2_{L^2(\Om)}\\
\le &M^{-2} \int^{\infty}_0 d\, \|E_{H_V}(\la) f\|^2_{L^2(\Om)}\\
\le & M^{-2} \|f\|^2_{L^2(\Om)}
\end{align*}
for any $f \in L^2(\Om)$. This proves \eqref{EQ:Lem5.2-1}. 

Since $R_{V,\theta}f \in \mathcal{D}(H_V)$ for any $f \in L^2(\Om)$, we can write
\begin{align*}
	\|\nabla R_{V,\theta}f\|^2_{L^2(\Om)}& = 
	\int_{\Om} \Big(\nabla R_{V,\theta}f \cdot \nabla R_{V,\theta}f 
	+ V |R_{V,\theta}f|^2 - V |R_{V,\theta}f|^2\Big)\, dx \\
	& = {\langle H_V R_{V,\theta}f, R_{V,\theta}f \rangle}_{L^2(\Om)} 
	- \int_{\Om} V |R_{V,\theta}f|^2 \,dx \\
	& = I + II. 
\end{align*}
Then we estimate the first term $I$ as 
\begin{align*}
	I& \le \int^{\infty}_0 \frac{\la}{(\theta \la + M )^2}\, d\|E_{H_V}(\la) f\|^2_{L^2(\Om)}\\
	& = \int^{\infty}_0 \theta^{-1} \cdot \frac{\theta \la}{\theta \la + M } \cdot \frac{1}{\theta \la + M } \, d\|E_{H_V}(\la) f\|^2_{L^2(\Om)}\\
	& \le \theta^{-1} M^{-1} \int^{\infty}_0 d\, \|E_{H_V}(\la) f\|^2_{L^2(\Om)}\\
	& \le \theta^{-1} M^{-1} \|f\|^2_{L^2(\Om)}.
\end{align*}
As to $II$, we have, by Lemma \ref{Lem:V}, 
\begin{align*}
	II& \le \int_{\Om} V_- |R_{V,\theta}f|^2 \,dx \\
	& \le \frac{\|V_-\|_{K_{d}(\Om)}\Gamma(d/2-1)}{4\pi^{d/2}} \int_{\Om} 
	|\nabla R_{V,\theta}f|^2 \,dx.
\end{align*}
Combining the previous estimates, 
we conclude that
\[
	\|\nabla R_{V,\theta}f\|^2_{L^2(\Om)} \le \theta^{-1} \left(1 - \frac{\|V_-\|_{K_{d}(\Om)}\Gamma(d/2-1)}{4\pi^{d/2}}\right)^{-1} M^{-1} \|f\|^2_{L^2(\Om)}
\]
for any $f \in L^2(\Om)$. This proves \eqref{EQ:Lem5.2-2}.
The proof of Lemma \ref{Lem:Lem5.2} is complete. 
\end{proof}

We are now in a position to prove Proposition \ref{prop:Prop5.1}. 

\begin{proof}[Proof of Proposition \ref{prop:Prop5.1}.]
Let us denote by $\mathscr{D} (\Omega)$ the totality of the test functions on $\Omega$, and 
by $\mathscr{D}' (\Omega)$ its dual space. 
We regard $X$ as $\mathscr{D} (\Omega)$ and $Y$ as $\mathscr{D}' (\Omega)$ in the definition of operator $\mathrm{Ad}$. 
Then we have, by Lemma \ref{Lem:Recursive-Rtheta1} in appendix \ref{App:AppendixB},
\begin{equation}\label{EQ:Rtheta-1}
\mathrm{Ad}^0(R_{V,\theta})=R_{V,\theta}, \quad \mathrm{Ad}^1(R_{V,\theta}) =-2\theta R_{V,\theta}\partial_{x_j}R_{V,\theta},
\end{equation}
\begin{equation}\label{EQ:Rtheta-k}
\mathrm{Ad}^k(R_{V,\theta}) = \theta \left\{-2 k \mathrm{Ad}^{k-1}(R_{V,\theta}) \partial_{x_j} R_{V,\theta} + k(k-1) \mathrm{Ad}^{k-2}(R_{V,\theta}) R_{V,\theta}\right\}
\end{equation}
for $k \ge 2$. 
Since $R_{V,\theta}$ and $\pa_{x_j}R_{V,\theta}$ 
are bounded on $L^2(\Om)$ by 
Lemma \ref{Lem:Lem5.2}, 
$\mathrm{Ad}^k(R_{V,\theta})$ are also 
bounded on $L^2(\Om)$ for any $k\ge0$. 
Before going to prove \eqref{EQ:Prop5.1}, 
we prepare the following estimates for 
$\mathrm{Ad}^k(R_{V,\theta})$: For any non-negative 
integer $k$, there exists a constant 
$C_k>0$ such that
\begin{equation}\label{EQ:L2-est-Rtheta-k}
\|\mathrm{Ad}^k(R_{V,\theta})\|_{\Bscr(L^2(\Om))} 
\le C_k \theta^{k/2}
\end{equation}
for any $\theta>0$. 
We prove \eqref{EQ:L2-est-Rtheta-k} by induction. 
For $k=0,1$, we have, by using \eqref{EQ:Rtheta-1} and Lemma \ref{Lem:Lem5.2},
\begin{align*}
\| \mathrm{Ad}^0(R_{V,\theta}) \|_{\Bscr(L^2(\Om))} = \| R_{V,\theta} \|_{\Bscr(L^2(\Om))} \le C_0,
\end{align*}
\begin{align*}
	\| \mathrm{Ad}^1(R_{V,\theta}) \|_{\Bscr(L^2(\Om))}& = 2 \theta \| R_{V,\theta} \partial_{x_j} R_{V,\theta} \|_{\Bscr(L^2(\Om))} \\
	& \le 2 \theta M^{-1} \cdot M^{-1/2}\left(1 - \frac{\|V_-\|_{K_{d}(\Om)}\Gamma(d/2-1)}{4\pi^{d/2}}\right)^{-1/2} \theta^{-1/2}\\ 
	& = C_1 \theta^{1/2}. 
\end{align*}
Let us suppose that \eqref{EQ:L2-est-Rtheta-k} is true for $k\in\{0,1,\ldots,\ell\}$. 
Combining identities \eqref{EQ:Rtheta-k} and estimates \eqref{EQ:Lem5.2-1} 
and \eqref{EQ:Lem5.2-2} from Lemma \ref{Lem:Lem5.2}, we get 
\eqref{EQ:L2-est-Rtheta-k} for $k=\ell+1$:
\begin{align*}
	&\left\| \mathrm{Ad}^{\ell+1}(R_{V,\theta})
	\right\|_{\Bscr(L^2(\Om))} \\
	=&\left\|\theta\left\{-2(\ell+1)
\mathrm{Ad}^{\ell}(R_{V,\theta}) \partial_{x_j} R_{V,\theta} + \ell(\ell+1)\mathrm{Ad}^{\ell-1}(R_{V,\theta}) R_{V,\theta}\right\} \right\|_{\Bscr(L^2(\Om))} \\
	\le& 2\ell(\ell+1) \theta \left\{ \| \mathrm{Ad}^{\ell}(R_{V,\theta}) \|_{\Bscr(L^2(\Om))} \|\partial_{x_j} R_{V,\theta}\|_{\Bscr(L^2(\Om))} \right.\\
	&\left. \qquad + \|\mathrm{Ad}^{\ell-1}(R_{V,\theta})\|_{\Bscr(L^2(\Om))} \|R_{V,\theta}\|_{\Bscr(L^2(\Om))} \right\} \\
	\le& C_{\ell+1} \theta \left\{ \theta^{\ell/2} \cdot \theta^{-1/2} + \theta^{(\ell-1)/2}
	 \right\} \\
	\le& C_{\ell+1} \theta^{(\ell+1)/2}.
\end{align*}
Thus 
\eqref{EQ:L2-est-Rtheta-k} is true for any $k \ge 0$. 

We prove \eqref{EQ:Prop5.1} also by induction. Clearly, \eqref{EQ:Prop5.1} is true for $k=0$.
As to the case $k=1$, by using  the estimate \eqref{EQ:L2-est-Rtheta-k} and the formula 
\eqref{EQ:Recursive-Rtheta2-1} from Lemma \ref{Lem:Recursive-Rtheta2} in 
appendix~\ref{App:AppendixB}:
\[
	\mathrm{Ad}^1(e^{-itR_{V,\theta}}) = - i \int^{t}_{0} e^{-isR_{V,\theta}} \mathrm{Ad}^1(R_{V,\theta}) e^{-i(t-s)R_{\theta, V}} \,ds,
\]
we have 
\begin{align*}
&\|\mathrm{Ad}^1(e^{-itR_{V,\theta}})
\|_{\Bscr(L^2(\Om))}\\
\le &
\int^{t}_{0} \|e^{-isR_{V,\theta}}
\|_{\Bscr(L^2(\Om))} \|\mathrm{Ad}^1(R_{V,\theta})
\|_{\Bscr(L^2(\Om))} \|e^{-i(t-s)R_{V,\theta}}
\|_{\Bscr(L^2(\Om))} \,ds \\
\le & C_1 \int^{t}_{0} \theta^{1/2} \,ds \\
\le	& C_1 \theta^{1/2} (1+t).
\end{align*}
Hence, \eqref{EQ:Prop5.1} is true for $k=1$.
Let us suppose that \eqref{EQ:Prop5.1} 
holds for $k\in\{0,1,\ldots,\ell\}$. 
Then, by using the estimate 
\eqref{EQ:L2-est-Rtheta-k} and the 
formula \eqref{EQ:Recursive-Rtheta2-2} from Lemma \ref{Lem:Recursive-Rtheta2}:
\begin{align*}
&\mathrm{Ad}^{\ell+1}(e^{-itR_{V,\theta}}) \\
=& - i \int^{t}_{0} 
\sum_{\ell_1+\ell_2+\ell_3=\ell}
\Gamma(\ell_1,\ell_2,\ell_3) 
\mathrm{Ad}^{\ell_1}
(e^{-isR_{\theta, V}}) 
\mathrm{Ad}^{\ell_2+1}(R_{V,\theta}) 
\mathrm{Ad}^{\ell_3}(e^{-i(t-s)R_{V,\theta}}) 
\,ds,\nonumber
\end{align*}
where constants $\Gamma(\ell_1,\ell_2,\ell_3)$ 
are trinomial coefficients: 
\[
\Gamma(\ell_1,\ell_2,\ell_3)
=\frac{\ell!}{\ell_1! \ell_2! \ell_3!},
\]
we estimate
\begin{align*}
& \|\mathrm{Ad}^{\ell+1}(e^{-itR_{V,\theta}})
\|_{\Bscr(L^2(\Om))} \\
\le& C_{\ell+1} \int^{t}_{0} 
\sum_{\ell_1+\ell_2+\ell_3=\ell}
\|\mathrm{Ad}^{\ell_1}(e^{-isR_{V,\theta}})
\|_{\Bscr(L^2(\Om))} 
\|\mathrm{Ad}^{\ell_2+1}(R_{V,\theta})
\|_{\Bscr(L^2(\Om))} \times \\
&\qquad \qquad \qquad \times 
\|\mathrm{Ad}^{\ell_3}(e^{-i(t-s)R_{V,\theta}})
\|_{\Bscr(L^2(\Om))} \,ds\\
\le& C_{\ell+1} \int^{t}_{0} 
\sum_{\ell_1+\ell_2+\ell_3=\ell}
\theta^{\ell_1 /2}(1+s)^{\ell_1} 
\cdot \theta^{(\ell_2+1) /2} 
\cdot \theta^{\ell_3 /2}
(1+t-s)^{\ell_3} \,ds\\
\le& C_{\ell+1} \theta^{(\ell+1)/2}
(1+t)^{\ell+1}. 
\end{align*}
Hence \eqref{EQ:Prop5.1} is true for $k=\ell+1$. 
Thus \eqref{EQ:Prop5.1} holds for any $k\ge0$. 
The proof of Proposition \ref{prop:Prop5.1} is complete.
\end{proof}


\section{Proof of Theorem \ref{Thm1.1}.}

In this section we shall prove 
Theorem \ref{Thm1.1}.
To begin with, let us introduce a family 
of operators which is useful to prove 
the theorem. 
For any non-negative integer $N$, we define a family $\mathscr{A}_N$ of operators as follows: 
We say that $A \in \mathscr{A}_N$ if $A \in \mathscr{B}(L^2(\Om))$ and
\begin{equation}\label{EQ:Def-A}
	{\vertiii{A}}_{N} := \sup_{n \in \mathbb{Z}^d} \big\| |\cdot - \theta^{1/2}n|^{N} A \chi_{C_{\theta}(n)}\big\|_{\mathscr{B}(L^2(\Om))} < \infty, 
\end{equation}
where $\chi_{C_{\theta}(n)}$ are the characteristic functions of 
cubes $C_{\theta}(n)$. \\

First, we prepare two lemmas. 
\begin{lem} \label{Lem6.1}
For any integer $N$ with $N>d/2$, 
there exists a constant $C(d, N)>0$ such that 
\begin{align}\label{EQ:KEY} 
	&\sum_{m \in \mathbb{Z}^d} \|\chi_{C_{\theta}(m)} A \chi_{C_{\theta}(n)} f\|_{L^2(\Om)} \\
	\le& C(d,N) \left( \|A\|_{\mathscr{B}(L^2(\Om))} 
	+ \theta^{-d/4} {\vertiii{A}}^{d/2N}_{N} \|A\|^{1-d/2N}_{\mathscr{B}(L^2(\Om))} \right)
	 \|\chi_{C_{\theta}(n)} f\|_{L^2(\Om)}
	 \nonumber
\end{align}
for all $n \in \mathbb{Z}^d$, $A \in \mathscr{A}_N$ and $f \in L^2(\Om)$.
\end{lem}
\begin{proof}
Let $n\in\Z^d$ be fixed. For any $\om > 0$, 
we write 
\begin{align*}
	& \sum_{m \in \mathbb{Z}^d} 
\big\|\chi_{C_{\theta}(m)} A \chi_{C_{\theta}(n)} f \big\|_{L^2(\Om)}\\ 
=& \sum_{|m - n| > \om} 
|\theta^{1/2}m - \theta^{1/2}n|^{- N} |\theta^{1/2}m - \theta^{1/2}n|^{N} \big\|\chi_{C_{\theta}(m)} A \chi_{C_{\theta}(n)} f \big\|_{L^2(\Om)}\\
	& \qquad  + 
	\sum_{|m - n| \le \om} \big\|\chi_{C_{\theta}(m)} A \chi_{C_{\theta}(n)} f \big\|_{L^2(\Om)}\\
	=:&I(n) + II(n).
\end{align*}
By using Schwarz inequality, we estimate 
$I(n)$ as 
\begin{align}\label{EQ:FAC}
I(n)
\le& \theta^{-N/2} \Big( \sum_{|m - n| > \om} 
|m - n|^{- 2N}\Big)^{1/2}\times \\
& \qquad \Big( \sum_{|m - n| > \om} 
|\theta^{1/2}m - \theta^{1/2}n|^{2N} 
\big\|\chi_{C_{\theta}(m)} A \chi_{C_{\theta}(n)} f \big\|^2_{L^2(\Om)}\Big)^{1/2}.
\nonumber
\end{align}
The first factor of \eqref{EQ:FAC} is estimated
as 
\begin{align}\label{EQ:T}
\sum_{|m - n| > \om} |m - n|^{-2N} =& \sum_{|m|>\om}|m|^{-2N} \\
\le& C(d,N)\om^{-2N+d}.
\nonumber
\end{align}
In fact, since $N>d/2$, the right member of \eqref{EQ:T} is 
estimated as  
\begin{align*}
\sum_{|m|>\om}|m|^{-2N}
\le
& \prod _{j = 1} ^d 
  \sum _{|m_j| > \frac{\om}{\sqrt{d}}} |m_j| ^{-2N/d}
\\ \nonumber
\le
& C(d,N) \prod _{j = 1} ^d 
  \sum _{|m_j| > \frac{\om}{\sqrt{d}}} (1 + |m_j| ) ^{-2N/d}
\\ \nonumber
\leq  
& C(d,N) \prod _{j = 1} ^d 
  \int_{\{\sigma>\frac{\om}{\sqrt{d}}\}} \sigma^{-2N/d} \, d\sigma 
\\ \nonumber
\leq  
& C(d,N) \prod _{j = 1} ^d 
  \om ^{ -2N /d + 1}
\\ \nonumber
= 
& C(d,N) \om ^{-2N + d},\nonumber
\end{align*} 
which implies \eqref{EQ:T}. 
As to the second factor of \eqref{EQ:FAC}, 
noting that 
\[
\frac{|\theta^{1/2}m - \theta^{1/2}n|}{2}\le 
|x - \theta^{1/2}n|
\]
for any $x\in C_\theta(m)$, 
we estimate as 
\begin{align*}
& \sum_{|m - n| > \om} |\theta^{1/2}m - \theta^{1/2}n|^{2N} 
\big\|\chi_{C_{\theta}(m)} A \chi_{C_{\theta}(n)} f \big\|^2_{L^2(\Om)} \\
=&\sum_{|m - n| > \om} |\theta^{1/2}
m - \theta^{1/2}n|^{2N}\int_{C_{\theta}(m)}
|A \chi_{C_{\theta}(n)} f|^2\,dx\\
\le&2^{2N}\sum_{|m - n| > \om}\int_{C_{\theta}(m)}\left|| x - \theta^{1/2}n|^{N}A \chi_{C_{\theta}(n)} f\right|^2\,dx.
\end{align*}
Moreover, by the definition 
\eqref{EQ:Def-A} of $\vertiii{A}_N$, we estimate as 
\begin{align*}
\sum_{|m - n| > \om}\int_{C_{\theta}(m)}\left|| x - \theta^{1/2}n|^{N}A \chi_{C_{\theta}(n)} f\right|^2\,dx
\le&\left\| \left| \cdot - \theta^{1/2}n
\right|^{N} A \chi_{C_{\theta}(n)} f\right\|^2_{L^2(\Om)}\\
\le&{\vertiii{A}}^2_{N}  \big\| \chi_{C_{\theta}(n)} f \big\|^2_{L^2(\Om)}.
\end{align*}
Hence, summarizing the above two estimates, we deduce that
\begin{equation}\label{EQ:CM}
\sum_{|m - n| > \om} |\theta^{1/2}m - \theta^{1/2}n|^{2N} 
\big\|\chi_{C_{\theta}(m)} A \chi_{C_{\theta}(n)} f \big\|^2_{L^2(\Om)}
\le 2^{2N} \vertiii{A}^2_N \big\| \chi_{C_{\theta}(n)} f \big\|^2_{L^2(\Om)}.
\end{equation}
Thus we find from \eqref{EQ:FAC}--\eqref{EQ:CM} that 
\begin{equation}\label{EQ:a}
I(n)\le C(d,N)\theta^{-N/2} \om^{-(N-d/2)} {\vertiii{A}}_{N} 
	 \big\| \chi_{C_{\theta}(n)} f \big\|_{L^2(\Om)}. 
\end{equation}

Let us turn to the estimation of $II(n)$. We estimate 
\[
II(n)\le \Big( \sum_{|m - n| \le \om} 1 \Big)^{1/2} \Big( \sum_{|m - n| \le \om} 
\big\|\chi_{C_{\theta}(m)} A \chi_{C_{\theta}(n)} f \big\|^2_{L^2(\Om)} \Big)^{1/2}.
\]
Since 
\[
\sum_{|m - n| \le \om} 1 \le 1 +\om^d,
\]
we deduce from the same argument as in $I(n)$
\begin{align}\label{EQ:b}
	I\hspace{-.1em}I(n)&\le (1 + \om^{d/2}) \Big( \sum_{|m - n| \le \om} 
	 \big\|\chi_{C_{\theta}(m)} A \chi_{C_{\theta}(n)} f \big\|^2_{L^2(\Om)} \Big)^{1/2} \\
	& \le (1 + \om^{d/2}) \big\|A \chi_{C_{\theta}(n)} f \big\|_{L^2(\Om)}\notag \\
	& \le (1 + \om^{d/2}) \|A\|_{\mathscr{B}(L^2(\Om))} \big\| \chi_{C_{\theta}(n)} f \big\|_{L^2(\Om)}.\notag
\end{align}
Combining the estimates \eqref{EQ:a} and \eqref{EQ:b}, we get 
\begin{align*}
	& \sum_{m \in \mathbb{Z}^d} 
	\|\chi_{C_{\theta}(m)} A \chi_{C_{\theta}(n)} f\|_{L^2(\Om)} \\
	\le& C(d,N) \Big\{ \theta^{-N/2} \om^{-(N-d/2)} {\vertiii{A}}_{N} 
	 + (1 + \om^{d/2})\|A\|_{\mathscr{B}(L^2(\Om))}\Big\} \big\| \chi_{C_{\theta}(n)} f \big\|_{L^2(\Om)}.
\end{align*}
Finally, taking $\om = ({\vertiii{A}}_{N} / \|A\|_{\mathscr{B}(L^2(\Om))})^{1/N} \cdot \theta^{-1/2}$, we obtain the required estimate
\eqref{EQ:KEY}. 
The proof of Lemma \ref{Lem6.1} is complete.
\end{proof}

\begin{lem}\label{Lem6.2}
Assume that the measurable potential 
$V = V_+ - V_-$ satisfies  
$V_{\pm} \in K_{d}(\Om)$, and that 
$V_-$ satisfies assumption \eqref{EQ:katonorm3} in Lemma \ref{Lem:comparison}. 
Let $N$ be a positive integer, and 
let $\psi \in \Sscr(\mathbb{R})$. 
Then $\psi(R_{V,\theta}) \in \mathscr{A}_N$. 
Furthermore, 
there exists a constant $C_{\psi} > 0$ such that 
\begin{equation}\label{Lem6.2-1}
	\|\psi(R_{V,\theta})\|_{\mathscr{B}(L^2(\Om))} \le C_{\psi}, \quad \forall \theta > 0,
\end{equation}
and 
there exists a constant $C_N > 0$ such that 
\begin{align}\label{Lem6.2-2}
	{\vertiii{\psi(R_{V,\theta})}}_{N}
	 \le C_N \theta^{N/2}\int^{\infty}_{-\infty} (1+|t|)^{N} |\hat{\psi}(t)| \,dt, \quad \forall \theta > 0. 
\end{align}
\end{lem}

\begin{proof}
The proof is based on the well-known formula:
\begin{equation}\label{Fourier}
	\psi(R_{V,\theta}) = (2 \pi)^{-1/2} \int^{\infty}_{-\infty} e^{-itR_{V,\theta}} \hat{\psi}(t) \,dt,
\end{equation}
where $\hat{\psi}$ is the Fourier transform of $\psi$ on $\mathbb{R}$.  
The estimate \eqref{Lem6.2-1} is an immediate consequence of the unitarity of $e^{-itR_{V,\theta}}$, the formula \eqref{Fourier} and $\psi \in \Sscr(\R)$. 

As to the estimate \eqref{Lem6.2-2}, 
applying the formula \eqref{Fourier}, 
we obtain 
\begin{align*}
	&{\vertiii{\psi(R_{V,\theta})}}_{N}\\
	=&\sup_{n \in \mathbb{Z}^d} \big\| |\cdot - \theta^{1/2}n|^{N} \psi(R_{V,\theta}) \chi_{C_{\theta}(n)}\big\|_{\Bscr(L^2(\Om))} \\
	\le& (2\pi)^{-1/2} \sup_{n \in \mathbb{Z}^d} \int^{\infty}_{-\infty} 
	\big\| |\cdot - \theta^{1/2}n|^{N} e^{-itR_{V,\theta}} \chi_{C_{\theta}(n)}\big\|_{\Bscr(L^2(\Om))} |\hat{\psi}(t)| \,dt. 
\end{align*}
Resorting to Lemma \ref{Lem:Recursive} for 
$A = B = x_j - \theta^{1/2} n_i$
 and $L = e^{-itR_{V,\theta}}$, we find from Proposition \ref{prop:Prop5.1} that
\begin{align*}
&\big\| |\cdot - \theta^{1/2}n|^{N} e^{-itR_{V,\theta}} \chi_{C_{\theta}(n)}\big\|_{\Bscr(L^2(\Om))} \\
\le& \sum^{N}_{k=0} C(N, k) \big\|\mathrm{Ad}^{k}(e^{-itR_{V,\theta}})\big\|_{\Bscr(L^2(\Om))} \big\||\cdot - \theta^{1/2}n|^{N-k} \chi_{C_{\theta}(n)}\big\|_{\Bscr(L^2(\Om))}\\
\le& \sum^{N}_{k=0}C(N, k) \theta^{k/2} 
(1+|t|)^{k} \theta^{(N-k)/2}; 
\end{align*}
thus we conclude that 
\begin{align*}
	{\vertiii{\psi(R_{V,\theta})}}_{N}\le \theta^{N/2} \sum^{N}_{k=0}C(N,k)
	\int^{\infty}_{-\infty}
	(1+|t|)^{k}|\hat{\psi}(t)| \,dt,
\end{align*}
which proves \eqref{Lem6.2-2}. The proof of Lemma \ref{Lem6.2} is finished. 
\end{proof}

We are now in a position to prove the main theorem.

\begin{proof}[Proof of Theorem \ref{Thm1.1}]
It suffices to show $L^1$-boundedness of $\va(\theta H_V)$. 
Let $\beta > d/4$ and $M>0$. 
Let us define $\psi \in \Sscr(\mathbb{R})$ as 
\[
	\psi (\mu) := \mu^{-\beta} \va(\mu^{-1} - M),\ \ \ \mu \in (0, 1/M].
\]
Then we can write
\begin{equation}\label{psi}
	\psi \big((\la + M)^{-1} \big) = \va(\la) (\la + M)^{\beta},\ \ \ \la \ge 0.
\end{equation}
Now we estimate, by H\"older's inequality and the definition of amalgam spaces $\ell^p(L^q)_\theta$, 
\begin{align*}
	\|\va(\theta H_V) f\|_{L^1(\Om)}& = \sum_{n\in\Z^d}\|\va(\theta H_V) f\|_{L^1(C_{\theta}(n))}\\
	\le&\sum_{n\in\Z^d}|C_\theta(n)|^{1/2}|\|\va(\theta H_V) f\|_{L^2(C_{\theta}(n))}\\
	\le& \theta^{d/4} \|\va(\theta H_V) f\|_{\ell^1(L^2)_{\theta}},
\end{align*}
where we used $|C_\theta(n)|^{1/2}=\theta^{d/4}$. 
The right member in the above inequality is estimates as 
\begin{align*}
\|\va(\theta H_V) f\|_{\ell^1(L^2)_{\theta}}=& \|\va(\theta H_V) (\theta H_V + M )^{\beta} 
R^{\beta}_{V,\theta} f\|_{\ell^1(L^2)_{\theta}}  \\
	=& \|\psi(R_{V,\theta}) R^{\beta}_{V,\theta} f\|_{\ell^1(L^2)_{\theta}}  \\
	\le& \sum_{n \in \mathbb{Z}^d} \sum_{m \in \mathbb{Z}^d} \big\|\chi_{C_{\theta}(m)} 
	\psi(R_{V,\theta}) \chi_{C_{\theta}(n)} R^{\beta}_{V,\theta} f \big\|_{L^2(\Om)},
\end{align*}
where we used \eqref{psi} in the second step. 
Resorting to Lemma \ref{Lem6.1} for $A$ and $f$ replaced by $\psi(R_{V,\theta})$ and 
$R^{\beta}_{V,\theta} f $, respectively, we estimate 
\begin{align*}
	&\sum_{m \in \mathbb{Z}^d} \|\chi_{C_{\theta}(m)} \psi(R_{V,\theta}) \chi_{C_{\theta}(n)} 
	R^{\beta}_{V,\theta} f\|_{L^2(\Om)} \\
	 \le& C \Big( \|\psi(R_{V,\theta})\|_{\Bscr(L^2(\Om))} 
	+ \theta^{-d/4} {\vertiii{\psi(R_{V,\theta})}}^{d/2N}_{N} \|\psi(R_{V,\theta})\|^{1-d/2N}_{\Bscr(L^2(\Om))} \Big)
	 \|\chi_{C_{\theta}(n)} R^{\beta}_{V,\theta} f\|_{L^2(\Om)}
\end{align*}
for any $N>d/2$.
Thus we obtain
\begin{align*}
	&\|\va(\theta H_V) f\|_{L^1(\Om)}\\[2mm]
	 \le& C \theta^{d/4} 
	\Big( \|\psi(R_{V,\theta})\|_{\Bscr(L^2(\Om))} + \theta^{-d/4} {\vertiii{\psi(R_{V,\theta})}}^{d/2N}_{N} \|\psi(R_{V,\theta})\|^{1-d/2N}_{\Bscr(L^2(\Om))} \Big) 
	\|R^{\beta}_{V,\theta} f\|_{\ell^1(L^2)_{\theta}}
\end{align*}
for any $N>d/2$.
Applying Theorem \ref{Thm4.1} and Lemma \ref{Lem6.2} to the above estimate, 
we conclude that 
\begin{align*}
	\|\va(\theta H_V) f\|_{L^1(\Om)}& \le C \theta^{d/4} \left\{ 1 + \theta^{-d/4} \cdot (\theta^{N/2})^{d/2N} \right\} \theta^{-d/4} \|f\|_{L^1(\Om)}\\[2mm]
	& \le C \|f\|_{L^1(\Om)},
\end{align*}
where the constant $C$ is independent of $\theta$. The proof of Theorem \ref{Thm1.1} is complete.
\end {proof}


\section{A final remark}
As a consequence of Theorems 
\ref{Thm1.1} and \ref{Thm4.1}, 
we have $L^p$-$L^q$-boundedness of $\va(\theta H_V)$. 
$L^p$-$L^q$-boundedness of $\va(\theta H_V)$ 
is useful to prove the embedding theorem for 
Besov spaces. 

\begin{prop}\label{Prop7.1}
Let $\va$ and $V$ be as in Theorem \ref{Thm1.1}. 
Then there exists a constant $C = C(d, \va) > 0$ such that for $1 \le p \le q \le \infty$,   
\[
	\|\va(\theta H_V)\|_{\mathscr{B}(L^p(\Om), L^q(\Om))} \le C {\theta}^{-(d/2)(1/p-1/q)},\quad \forall \theta > 0.
\]
\end{prop}

\begin{proof}
Let us define $\tilde{\va}\in \mathscr{S}(\mathbb{R})$ as
\[
	\tilde{\va}(\la) = (\la + M)^{\beta} \va(\la), \quad \la \ge 0. 
\]
By Theorems \ref{Thm1.1} and \ref{Thm4.1}, 
for $1 \le p \le q \le \infty$ and 
$\beta > (d/2)(1/p -1/q)$, we estimate 
\begin{align*}
	\|\va(\theta H_V)\|_{\mathscr{B}(L^p(\Om), L^q(\Om))}& = 
	\|\va(\theta H_V) (\theta H_V + M)^{\beta} (\theta H_V + M)^{-\beta}
	\|_{\mathscr{B}(L^p(\Om), L^q(\Om))}\\[2mm]
	& \le \|\tilde{\va}(\theta H_V)\|_{\mathscr{B}(L^q(\Om))} 
	\|(\theta H_V + M)^{-\beta}\|_{\mathscr{B}(L^p(\Om), L^q(\Om))}\\[2mm]
	& \le C {\theta}^{-(d/2)(1/p-1/q)}. 
\end{align*}
The proof of Proposition \ref{Prop7.1} is complete.
\end{proof} 

\appendix
\section{
(The Young inequality)
} \label{App:AppendixA}
In this appendix we introduce the Young inequality for scaled amalgam spaces.

\begin{lem}\label{Lem:Young}
Let $d\ge1$, and let $1 \le p,p_1,p_2,q,q_1,q_2 \le \infty$ be such that 
\[
\text{$\frac{1}{p_1} +\frac{1}{p_2} - 1 = \frac{1}{p}$ \quad and \quad 
$\frac{1}{q_1} + \frac{1}{q_2} - 1 = \frac{1}{q}$.}
\]
If $f \in \ell^{p_1}(L^{q_1})_{\theta}(\R^d)$ and $g \in \ell^{p_2}(L^{q_2})_{\theta}(\R^d)$, 
then $\int_{\R^d}f(\cdot-y) g(y)\, dy \in \ell^{p}(L^{q})_{\theta}(\R^d)$ and 
\begin{equation}\label{EQ:Young}
	\left\|\int_{\R^d}f(\cdot-y) g(y)\, dy
	\right\|_{\ell^p(L^q)_{\theta}(\R^d)} \le 3^d \|f\|_{\ell^{p_1}(L^{q_1})_{\theta}(\R^d)} 
	\|g\|_{\ell^{p_2}(L^{q_2})_{\theta}(\R^d)}.
\end{equation}
\end{lem}

For the proof of Lemma \ref{Lem:Young}, see Fournier and Stewart \cite{FS}.

\section{
(Recursive formula of operators)
} \label{App:AppendixB}
In this appendix 
we introduce some formulas on the operator 
$\mathrm{Ad}$. 

\begin{lem}[Lemma 3.1 from \cite{JN2}]\label{Lem:Recursive}
Let $X$ and $Y$ be topological vector spaces, 
and let $A$ and $B$ be continuous 
linear operators from $X$ and $Y$ into 
themselves, respectively. If $L$ is a 
continuous linear operator from $X$ into $Y$, 
then there exists a set of constants 
$\left\{ C(n, m) \big|\, n \ge 0,\, 
0 \le m \le n \right\}$ such that 
\begin{equation}\label{EQ:Recursive}
B^n L = \sum^{n}_{m=0} C(n, m) \mathrm{Ad}^m(L) A^{n-m}.
\end{equation}
\end{lem}

We shall derive two kind of recursive formulas of operator 
\begin{equation}\label{EQ:RE}
R_{V,\theta}=(\theta H_V+M)^{-1},
\end{equation}
where $M$ is a certain large constant.
Hereafter we put
\[
X=\Dscr(\Om), \quad Y=\Dscr'(\Om),
\]
where 
we denote by $\mathscr{D} (\Omega)$ the totality of the test functions on $\Omega$, and 
by $\mathscr{D}' (\Omega)$ its dual space, 
and we take
\[
A=B=x_j-\theta^{1/2}n_j 
\quad \text{for some $j\in\{1,\cdots,d\}$}. 
\]

\begin{lem}\label{Lem:Recursive-Rtheta1}
Let $V$ be a measurable function on $\Omega$ 
such that $H_V$ is a self-adjoint operator 
on $L^2(\Omega)$ whose domain is given by 
\[
\mathcal{D}(H_V) = \big\{ u \in H^1_0(\Om)\, \big| \, H_V u \in L^2(\Om) \big\}.
\]
Let $M$ be an element of resolvent set of 
$-\theta H_V$, and let us denote by $R_{V,\theta}$ the resolvent operator defined by \eqref{EQ:RE}.
Then the sequence 
$\{\mathrm{Ad}^k(R_{V,\theta})\}_{k=0}^\infty$ 
of operators satisfies 
the following recursive formula{\rm :}
\begin{equation}\label{EQ:Recursive-Rtheta1}
\mathrm{Ad}^0(R_{V,\theta})=R_{V,\theta}, \quad \mathrm{Ad}^1(R_{V,\theta}) =-2\theta R_{V,\theta}\partial_{x_j}R_{V,\theta},
\end{equation}
and for $k\ge2$, 
\begin{equation}\label{EQ:Recursive-Rthetak}
\mathrm{Ad}^k(R_{V,\theta}) = \theta \left\{-2 k \mathrm{Ad}^{k-1}(R_{V,\theta})
 \partial_{x_j} R_{V,\theta} + k(k-1) \mathrm{Ad}^{k-2}(R_{V,\theta}) R_{V,\theta}\right\}. 
\end{equation}
\end{lem}

\begin{proof}
When $k=0$, the first equation in 
\eqref{EQ:Recursive-Rtheta1}
is trivial. Hence it is sufficient to prove the case when $k>0$. 
For the sake of simplicity, we perform 
a formal argument without considering the 
domain of operators. 
The rigorous argument is given in the final part. 

Let us introduce the generalized 
binomial coefficients $\Gamma(k,m)$ 
as follows:
\begin{equation*}
\Gamma(k,m) =
\begin{cases}
	\displaystyle{\frac{k!}{(k-m)!m!}},\ \ \ &k \ge m \ge 0,\\
	\ 0,\ \ \ &k < m \text{ or } k < 0.
\end{cases}
\end{equation*}
Once the following recursive formula 
is established:  
\begin{equation}\label{EQ:bbb}
\mathrm{Ad}^k(R_{V,\theta}) = - \sum^{k-1}_{m=0} \Gamma(k,m) \mathrm{Ad}^m(R_{V,\theta}) \mathrm{Ad}^{k-m}(\theta H_V) R_{V,\theta},\quad k=1,2,\cdots,
\end{equation}
identities \eqref{EQ:Recursive-Rtheta1} and \eqref{EQ:Recursive-Rthetak} are an 
immediate consequence of \eqref{EQ:bbb}, 
since 
\[
\mathrm{Ad}^1(\theta H_V) = 2 \theta \partial_{x_j},
\quad \mathrm{Ad}^2(\theta H_V) 
= -2\theta,\quad \mathrm{Ad}^k(\theta H_V) = 0,
\quad k \ge 3.
\]
Hence, all we have to do is to 
prove \eqref{EQ:bbb}. 
We proceed the argument by induction. 
For $k=1$, it can be readily checked that 
\begin{align*}
	 \mathrm{Ad}^1(R_{V,\theta})=& x_jR_{V,\theta} - R_{V,\theta} x_j \\
	 =&R_{V,\theta}(\theta H_V + M )x_jR_{V,\theta} - R_{V,\theta} x_j (\theta H_V + M )R_{V,\theta} \\
	 =&R_{V,\theta}\left(\theta H_Vx_j-x_j \cdot \theta H_V \right)R_{V,\theta}\\
	 =& - R_{V,\theta} \mathrm{Ad}^1(\theta H_V) R_{V,\theta}\\
	 =&- \Gamma(1,0) \mathrm{Ad}^0(R_{V,\theta}) \mathrm{Ad}^1(\theta H_V) R_{V,\theta}.
\end{align*}
Hence \eqref{EQ:bbb} is true for $k=1$. 
Let us suppose that \eqref{EQ:bbb} 
holds for $k =1, \ldots, \ell$. 
Writing
\begin{equation}\label{EQ:F1}
\mathrm{Ad}^{\ell+1}(R_{V,\theta})=x_j \mathrm{Ad}^{\ell}(R_{V,\theta})-\mathrm{Ad}^{\ell}
(R_{V,\theta})x_j,
\end{equation}
we see that the first term becomes 
\begin{align*}
&x_j \mathrm{Ad}^{\ell}(R_{V,\theta})\\
=&x_j \bigg\{-\sum^{\ell-1}_{m=0}
\Gamma(\ell,m) \mathrm{Ad}^m(R_{V,\theta}) 
\mathrm{Ad}^{\ell-m}(\theta H_V) 
\bigg\}R_{V,\theta}
\\
=&- \sum^{\ell-1}_{m=0} \Gamma(\ell,m)\left\{
\mathrm{Ad}^{m+1}(R_{V,\theta})\mathrm{Ad}^{\ell-m}(\theta H_V)+\mathrm{Ad}^m(R_{V,\theta}) \mathrm{Ad}^{\ell-m+1}(\theta H_V)\right\}R_{V,\theta}\\
& - \sum^{\ell-1}_{m=0} \Gamma(\ell,m) \mathrm{Ad}^m(R_{V,\theta}) \mathrm{Ad}^{\ell-m}(\theta H_V) x_jR_{V,\theta}\\
=&:I_1+I_2.
\end{align*}
Here $I_1$ is written as 
\begin{align*}
I_1
=&- \sum^{\ell}_{m=1} \Gamma(\ell,m-1) \mathrm{Ad}^{m}(R_{V,\theta})\mathrm{Ad}^{\ell-m+1}(\theta H_V)R_{V,\theta}\\
& -\sum^{\ell-1}_{m=0} \Gamma(\ell,m) \mathrm{Ad}^m(R_{V,\theta}) \mathrm{Ad}^{\ell-m+1}(\theta H_V)R_{V,\theta}\\
=&- \sum^{\ell}_{m=0} \Gamma(\ell,m-1) \mathrm{Ad}^{m}(R_{V,\theta})\mathrm{Ad}^{\ell+1-m}(\theta H_V)R_{V,\theta}\\
&- \sum^{\ell}_{m=0} \Gamma(\ell,m) \mathrm{Ad}^m(R_{V,\theta}) \mathrm{Ad}^{\ell-m+1}(\theta H_V)+\mathrm{Ad}^{\ell}(R_{V,\theta})\mathrm{Ad}^1(\theta H_V)R_{V,\theta}\\
=&- \sum^{\ell}_{m=0} \Gamma(\ell+1,m) \mathrm{Ad}^m(R_{V,\theta}) \mathrm{Ad}^{\ell+1-m}(\theta H_V)+\mathrm{Ad}^{\ell}(R_{V,\theta})\mathrm{Ad}^1(\theta H_V)R_{V,\theta}, 
\end{align*}
where we used 
\[
\Gamma(\ell,m-1)+\Gamma(\ell,m)=
\Gamma(\ell+1,m)
\]
in the last step. As to $I_2$, we write as 
\begin{align*}
I_2=&- \bigg\{\sum^{\ell-1}_{m=0} \Gamma(\ell,m) \mathrm{Ad}^m(R_{V,\theta}) \mathrm{Ad}^{\ell-m}(\theta H_V) R_{V,\theta}\bigg\} (\theta H_V + M ) x_j R_{V,\theta}\\
	=& \mathrm{Ad}^{\ell}(R_{V,\theta}) (\theta H_V + M ) x_jR_{V,\theta}.
\end{align*}
Hence, summarizing the previous equations,
we get 
\begin{align*}
x_j\mathrm{Ad}^\ell(R_{V,\theta})
=&
- \sum^{\ell}_{m=0} \Gamma(\ell+1,m) 
\mathrm{Ad}^m(R_{V,\theta}) 
\mathrm{Ad}^{\ell+1-m}(\theta H_V)\\
&+\mathrm{Ad}^{\ell}(R_{V,\theta}) \left\{\mathrm{Ad}^1(\theta H_V)+
(\theta H_V + M ) x_j\right\}R_{V,\theta}.
\end{align*}
Therefore, going back to \eqref{EQ:F1}, 
and noting 
\[
\mathrm{Ad}^1(\theta H_V)+ (\theta H_V + M ) x_j
= x_j(\theta H_V + M ),
\]
we conclude that 
\begin{align*}
\mathrm{Ad}^{\ell+1}(R_{V,\theta})
=&- \sum^{\ell}_{m=0} 
\Gamma(\ell+1,m) \mathrm{Ad}^m(R_{V,\theta}) 
\mathrm{Ad}^{\ell+1-m}(\theta H_V)\\ 
& +\mathrm{Ad}^{\ell}(R_{V,\theta})\left\{\mathrm{Ad}^1(\theta H_V)+ (\theta H_V + M ) x_j\right\}R_{V,\theta}
-\mathrm{Ad}^{\ell}(R_{V,\theta})x_j\\
=&- \sum^{\ell}_{m=0} 
\Gamma(\ell+1,m) \mathrm{Ad}^m(R_{V,\theta}) 
\mathrm{Ad}^{\ell+1-m}(\theta H_V)\\ 
& \qquad +\mathrm{Ad}^{\ell}(R_{V,\theta}) x_j(\theta H_V + M ) R_{V,\theta}
-\mathrm{Ad}^{\ell}(R_{V,\theta})x_j\\
=&- \sum^{\ell}_{m=0} \Gamma(\ell+1,m) \mathrm{Ad}^m(R_{V,\theta}) \mathrm{Ad}^{\ell+1-m}(\theta H_V).
\end{align*}
Hence \eqref{EQ:bbb} is true for $k=\ell+1$. \\

The above proof is formal in the sense that the domain of operators is not taken into account in the argument.  
In fact, even for $f\in C^\infty_0(\Om)$, each $x_jR_{V,\theta} f$ does not necessarily 
belong to the domain of $H_V$, 
since we only know the fact that
$$R_{V,\theta} f \in \mathcal D (H_V) = 
\{u\in H_0^1 (\Omega)\, |\, H_V u \in L^2 (\Omega) \}.
$$
Therefore, we should perform the argument 
by using a duality pair 
${}_{\Dscr'(\Omega)}\langle \cdot, 
\cdot\rangle_{\Dscr(\Omega)}$
of $\mathscr{D}^\prime(\Omega)$ and $\mathscr{D}(\Omega)$
in a rigorous way. 
We may prove the lemma 
only for $k=1$. For, as to the case $k>1$, 
the argument is done in a similar manner. 
Now we write
\begin{align*}
{}_{\Dscr'(\Omega)}\langle \mathrm{Ad}^1(R_{V,\theta})f,g\rangle_{\Dscr(\Omega)} &= \langle R_{V,\theta} f, x_j g\rangle_{L^2(\Om)} - \langle x_jf, R_{V,\theta} g\rangle_{L^2(\Om)} \\
	&=:I-II
\end{align*}
for $f,g\in C^\infty_0(\Om)$.
Since $R_{V,\theta} f, R_{V,\theta} g\in H^1_0(\Om)$, 
there exist two sequences $\{f_n\}_n$, 
$\{g_m\}_m$ in $C^\infty_0(\Om)$ such that 
\[
f_n\to R_{V,\theta} f \quad \text{and}\quad 
g_m\to R_{V,\theta} g \quad 
\text{in }H^1(\Om) \quad (n,m\to\infty). 
\]
Hence we obtain by 
$x_j f_n , x_j g_m \in C_0^\infty (\Omega) $,
\begin{align*}
I=&\lim_{n\to\infty} \left\langle f_n, x_jg\right\rangle_{L^2(\Om)}\\
=&\lim_{n\to\infty} \left\langle x_jf_n,(\theta H_V+M)
R_{V,\theta} g\right\rangle_{L^2(\Om)}\\
=&\lim_{n\to\infty} \left\{ \theta\left\langle \nabla(x_jf_n), \nabla R_{V,\theta} g\right\rangle_{L^2(\Om)} 
+\langle (\theta V+M)x_jf_n,
R_{V,\theta} g \rangle_{L^2(\Omega)}
\right\}\\
=&\lim_{n,m\to\infty} \left\{ \theta\left\langle \nabla(x_jf_n), \nabla g_m\right\rangle_{L^2(\Om)}
+\langle (\theta V+M)x_jf_n,g_m\rangle_{L^2(\Omega)}
\right\}\\
=&\lim_{n,m\to\infty} \left\{\theta \left\langle f_n, \pa_{x_j} g_m\right\rangle_{L^2(\Om)}
+\theta\left\langle x_j\nabla f_n, \nabla g_m\right\rangle_{L^2(\Om)} 
+\langle (\theta V+M)x_jf_n,g_m\rangle_{L^2(\Omega)}
\right\}
\end{align*}
and 
\begin{align*}
II=&\lim_{m\to\infty} \left\langle x_jf, 
g_m\right\rangle_{L^2(\Om)}\\
=&\lim_{m\to\infty} \left\langle 
(\theta H_V+M)R_{V,\theta} f,
x_jg_m\right\rangle_{L^2(\Om)}\\
=&\lim_{m\to\infty} \left\{ \theta\left\langle \nabla R_{V,\theta} f, \nabla (x_jg_m)\right\rangle_{L^2(\Om)} 
+\langle (\theta V+M)x_jR_{V,\theta} f,g_m\rangle_{L^2(\Omega)}
\right\}\\
=&\lim_{n,m\to\infty} \left\{\theta\left\langle \nabla f_n, \nabla (x_jg_m)\right\rangle_{L^2(\Om)} 
+\langle (\theta V+M)x_jf_n,g_m\rangle_{L^2(\Omega)}
\right\}\\
=&\lim_{n,m\to\infty} \left\{\theta\left\langle \pa_{x_j} f_n,  g_m\right\rangle_{L^2(\Om)}
+\theta\left\langle x_j\nabla f_n,\nabla g_m
\right\rangle_{L^2(\Om)} 
+\langle (\theta V+M)x_jf_n,g_m\rangle_{L^2(\Omega)}
\right\}.
\end{align*}
Then, combining the above equations, we deduce that 
\begin{align*}
	{}_{\Dscr'(\Omega)}\langle 
	\mathrm{Ad}^1(R_{V,\theta})f,g\rangle_{\Dscr(\Omega)} =& \lim_{n,m\to\infty} \theta\left\{ \langle f_n, \pa_{x_j} g_m\rangle_{L^2(\Om)}-\langle \pa_{x_j} f_n,  g_m\rangle_{L^2(\Om)} \right\}\\
	=&\lim_{n,m\to\infty}\theta \langle -2 \pa_{x_j} f_n,  g_m\rangle_{L^2(\Om)}\\
	=&\langle -2\theta\pa_{x_j} R_{V,\theta} f,
	R_{V,\theta} g\rangle_{L^2(\Om)}\\
	=&\langle -2\theta R_{V,\theta} \pa_{x_j} R_{V,\theta} f,
	g\rangle_{L^2(\Om)}
\end{align*}
for any $f,g\in C^\infty_0(\Omega)$.
Thus \eqref{EQ:Recursive-Rtheta1} is valid in 
a distributional sense. 
In a similar way, 
\eqref{EQ:Recursive-Rthetak} 
can be also shown in a distributional sense. 
The proof of Lemma \ref{Lem:Recursive-Rtheta1} 
is finished. 
\end{proof}

\begin{lem}\label{Lem:Recursive-Rtheta2}
Assume that $V$ satisfies the same assumption as in Lemma \ref{Lem:Recursive-Rtheta1}. 
Let $A$, $B$ and $L$ be as in Lemma \ref{Lem:Recursive-Rtheta1}. 
Then 
the following formula holds for each $t>0$\,{\rm :}
\begin{equation}\label{EQ:Recursive-Rtheta2-1}
	\mathrm{Ad}^1(e^{-itR_{V,\theta}}) = - i \int^{t}_{0} e^{-isR_{V,\theta}} 
	\mathrm{Ad}^1(R_{V,\theta}) e^{-i(t-s)R_{V,\theta}} \,ds. 
\end{equation}
Furthermore, the following 
formulas hold for $k>1${\rm :}
\begin{align}\label{EQ:Recursive-Rtheta2-2}
	&\mathrm{Ad}^{k+1}(e^{-itR_{V,\theta}}) \\
	=& - i \int^{t}_{0} \sum_{k_1 + k_2 + k_3 = k}
	\Gamma(k_1,k_2,k_3) \mathrm{Ad}^{k_1}(e^{-isR_{V,\theta}}) \mathrm{Ad}^{k_2+1}(R_{V,\theta}) \mathrm{Ad}^{k_3}(e^{-i(t-s)R_{V,\theta}}) \,ds,\notag
\end{align}
where the constants $\Gamma(k_1,k_2,k_3)$
$(k_1,k_2,k_3 \ge0)$ 
are trinomial coefficients{\rm :}
\[
\Gamma(k_1,k_2,k_3)=\frac{k!}{k_1!k_2!k_3!}.
\]
\end{lem}

\begin{proof}
It is sufficient to prove the lemma without taking account of the domain 
of operators as in the proof of Lemma \ref{Lem:Recursive-Rtheta1}. 
We write 
\begin{align*}
\mathrm{Ad}^1(e^{-itR_{V,\theta}}) 
& = x_j e^{-itR_{V,\theta}} 
- e^{-itR_{V,\theta}} x_j\\
& =  - \int^t_0 \frac{d}{ds} 
\big(e^{-isR_{V,\theta}} x_j 
e^{-i(t-s)R_{V,\theta}} \big)\, ds \\
& =  -i \int^t_0 e^{-isR_{V,\theta}} 
(x_j R_{V,\theta} - R_{V,\theta} x_j) 
e^{-i(t-s)R_{V,\theta}} \, ds\\
& =  -i \int^t_0 e^{-isR_{V,\theta}} 
\mathrm{Ad}^1 (R_{V,\theta}) 
e^{-i(t-s)R_{V,\theta}} \, ds.
\end{align*}
This proves \eqref{EQ:Recursive-Rtheta2-1}. 
The proof of \eqref{EQ:Recursive-Rtheta2-2} is 
performed by induction argument.
So we may omit the details. 
The proof of Lemma \ref{Lem:Recursive-Rtheta2} is complete.  
\end{proof}

\end{document}